\documentclass[a4paper, 10pt]{amsart}

\usepackage{amsmath}
\usepackage{amsfonts}
\usepackage{amssymb}
\usepackage{amsthm}
\usepackage{graphicx}
\usepackage[TABBOTCAP]{subfigure}
\usepackage{algorithm}
\usepackage{algorithmic}

\newcommand{\rr}{\mathbb{R}}
\newcommand{\nn}{\mathbb{N}}
\renewcommand{\.}{\cdot}
\newcommand{\e}{\varepsilon}
\newcommand{\test}{\varphi}
\newcommand{\diver}{\text{div\,}}
\newcommand{\n}{\nabla}
\newcommand{\Y}{\mathcal{Y}}
\newcommand{\J}{\mathcal{J}}
\newcommand{\xe}{\Xi_{\varepsilon} }
\newcommand{\xnull}{\Xi_0}
\newcommand{\abs}[1]{\left|#1\right|}

\DeclareMathOperator*{\essinf}{ess\,inf}

\newtheorem{theorem}{Theorem}[section]
\newtheorem{proposition}[theorem]{Proposition}
\newtheorem{lemma}[theorem]{Lemma}
\newtheorem{corollary}[theorem]{Corollary}

\theoremstyle{definition}
\newtheorem{definition}[theorem]{Definition}
\newtheorem{assumption}[theorem]{Assumption}
\newtheorem{problem}{Problem}

\theoremstyle{remark}
\newtheorem{remark}[theorem]{Remark}
\newtheorem{example}[theorem]{Example}

\renewcommand{\algorithmicrequire}{\textbf{\textsf{input:}}}
\renewcommand{\algorithmicensure}{\textbf{\textsf{output:}}}

\begin{document}

\title[Optimal Design of QDD in the semi-classical limit]{The semi-classical limit of an optimal design problem for the stationary quantum drift-diffusion model}
\author{Ren\'e Pinnau, Sebastian Rau, Florian Schneider and Oliver Tse}
\begin{abstract}
 We consider an optimal semiconductor design problem for the quantum drift diffusion (QDD) model in the semiclassical limit. The design question is formulated as a PDE constrained optimal control problem, where the doping profile acts as control variable. The existence of minimizers for any scaled Planck constant allows for the investigation of the corresponding sequence. Using the concepts of Gamma-convergence and equi-coercivity we can show the convergence of minima and minimizers. Due to the lack of uniqueness for the state system and optimization problem, it was necessary to establish a new result for the QDD model ensuring the existence of a sequence of quantum solutions converging to an isolated classical solution. As a by-product, we obtain new insights into the regularizing property of the quantum Bohm potential. Finally, we present the numerical optimization of a MESFET device underlining our analytical results.  
\end{abstract}  

\maketitle

\section{Introduction}
 Semiconductors are part of most electrical devices and are used to manufacture different components such as transistors, photovoltaic cells and diodes. Transistors switch or amplify electrical signals, while photovoltaic cells convert the energy of light, into electricity, and a diode lets an electrical signal pass only in one direction, which is utilized in light-emitting-diodes (LEDs), see \cite{Yacobi, Muller2003, Sze1985} for an overview.

 Since electrical properties of a semiconductor are in between those of an isolator and a conductor, they can be manipulated by implanting impurities into the semiconductor crystal, i.e., by embedding crystals that either have a lack (so-called \textit{p-doping}) or a surplus (\textit{n-doping}) of electrons. Therewith one can influence the conductivity and crystalline structure of the semiconductor. 
The task of an engineer in the semiconductor industry is to construct a semiconductor with the desired properties at reasonable development costs. For instance, one could desire to obtain a low leakage current in the off-state to maximize battery lifetime, and a high driving current in the on-state \cite{SSP98,SSP99}. Due to the ongoing miniaturization of semiconductor devices, the need to include quantum effects arises. Unfortunately, the direct simulation of microscopic quantum models are computationally expensive, and as such, optimization with these models becomes a complicated task. Nevertheless, first steps in this direction were done (see, e.g., \cite{borzi12} and the references therein).

So far, the mathematical semiconductor design has focussed on classical macroscopic semiconductor model hierarchy, i.e., the  stationary drift-diffusion equations \cite{BurPin, HinPin, HinPin01}, as well as on the energy transport model \cite{DraPin,drago2013}. These macroscopic models, in combination with modern optimization algorithms, proved to be a reliable tool for the optimal design of semiconductor doping profiles. 

Based on this experience it is natural to also exploit the hierarchy of macroscopic semiconductor quantum models. Here, one uses the quantum drift diffusion model \cite{AbdUnt,AncIaf, pinnaureviewTQDD,pinnaunumTQDD} as well as the quantum energy transport model \cite{juengelquasiQHD,degondgallegomehats2,jungel2011simplified}. First analytical and numerical results concerning corresponding optimization problems can be found \cite{UntVol,Schn11,PinRauSch,raudiss}. 

Since the quantum drift-diffusion model only differs from the drift-diffusion model by an additional term $\e^2 \frac{\Delta \sqrt{n}}{\sqrt{n}}$, known as the {\em Bohm potential} , with $n$ being the electron density and $\e^2$ denoting the squared scaled Planck constant, it is natural to investigate the so-called semi-classical limit $\e \rightarrow 0$, i.e., the transition from quantum to classical regime. This limit was first investigated in \cite{AbdUnt}. An extension of this model to self-gravitating particles is given in \cite{pinnautse2014}. These results state that every sequence of solutions to the quantum drift-difusion equations contains a subsequence that converges weakly in $H^1(\Omega)$ towards a solution of the drift-diffusion equations. 

Naturally, the question arises if such a result can be also expected for the corresponding semiconductor design problem. The numerical results in \cite{raudiss,Schn11} indicate that this limit should also hold in this case. Nevertheless, one encounters severe analytical problems, since neither the quantum drift diffusion model nor the corresponding optimal design is uniquely solvable.

For this cause, we will use the concept of $\Gamma$--convergence as described in \cite{DalMas}. Our strategy is the following: We begin by showing that the family of characteristic functions over the set of solutions to the quantum drift-diffusion equations $\Gamma$--converge to the characteristic function of a set of well-behaved solutions to the drift-diffusion equations. By including the characteristic functions into the cost functionals and assuming a special structure of the cost functionals, we obtain the $\Gamma$--convergence of the family of functionals in the semi-classical limit. Together with the equi-coercivity of the functionals, we then obtain the convergence of minima and minimizers.

The crucial part is here the so-called limsup-condition, which requires to find for a classical solution, a sequence of solutions to the quantum model converging to this classical solution. This cannot be deduced from the results in \cite{AbdUnt} and is nontrivial due to the lack of uniqueness. Here, we show for the first time that for an isolated classical solution, there exists, indeed, a sequence of quantum solutions converging to it. This gives also new insights into the regularizing behaviour of the quantum Bohm potential (cf.~\cite{pinnautse2014}, see also \cite{tse2014}).  

The manuscript is organized as follows. In Section 1 we introduce the state system and state the corresponding existence and regularity results. Then, the optimal control problem is formulated in Section 2 and the existence of minimizers is shown. The basic concept of $\Gamma$--convergence is briefly outlined in Section~3, followed by its application to our optimal control problem in Section 4. This allows us to show, in Section 5, the convergence of minimizers and minima. Finally, these analytical results are underlined by several numerical examples in Section 6 and concluding remarks are given in Section 7. In the Appendices we present some technical results necessary for the proofs in previous sections.

\section{State System}

Before stating the quantum drift-diffusion equations, we impose the following assumptions on the domain, its boundary and the boundary data (cf.~\cite{AbdUnt,HinPin,PinUnt,UntVol}).

\begin{assumption} 
\label{assdomainGAMMA}
\begin{enumerate}
\item Let $\Omega \subset \rr^d$, $d\in \{1,2,3\}$ be a Lipschitz domain. The boundary $\partial \Omega$ splits into two disjoint parts $\Gamma_N$, $\Gamma_D$. The set $\Gamma_D$ is assumed to closed and have non-vanishing ($d\!-\!1$)-dimensional Lebesgue measure.
\item Let $\Gamma_D= \bigcup^M_{i=1} \Gamma_D^i$, $M\geq 2$ and $\text{dist}\,(\Gamma_D^{i}, \Gamma_D^{j} ) > 0$ for $i \neq j$.
\item Let $\rho_D$, $V_D$, $S_D\in H^1(\Omega) \cap L^{\infty}(\Omega)$ with $\inf_\Omega \rho_D > 0$.
\end{enumerate}
\end{assumption}
\begin{remark}
The assumption above requires that any two Dirichlet boundaries $\Gamma^{i}_D$ and $\Gamma^{j}_D$, $i\ne j$ are separated by an insulating part. From the physical point of view this requirement is reasonable since it prevents short-circuiting.
\end{remark}

The quantum drift-diffusion equations are given by
\begin{subequations}\label{QDD}
\begin{align}
-\e^2 \Delta \rho  + \rho\big(h(\rho^2) +V-S\big) & = 0, \\
\label{eqV}
-\lambda^2 \Delta V & = \rho^2 - C,\\
-\diver(\rho^2 \nabla S) & = 0 .
\end{align}
on $\Omega$ with the boundary data
\begin{align}
\rho = \rho_D, \quad V = V_D, \quad S = S_D \quad \text{ on } \Gamma_D,  \\
\partial_\nu \rho =  \partial_\nu V = \rho^2\partial_\nu S = 0 \qquad \text{ on } \Gamma_N,
\end{align}
\end{subequations}
where $\nu$ denotes the outer normal, the variable $\rho$ denotes the square root of the {\em electron density} $n$, i.e., $\rho :=\sqrt{n}$, $S$ the {\em quasi-Fermi potential} and $V$ the {\em electrostatic potential} induced by the electron density. Throughout this paper, the doping profile $C$ is our control parameter. 

\begin{remark}\label{boundarydata}
Physically relevant boundary conditions were derived and used, e.g., in \cite{deFalco05,Mar04,UntVol} to ensure charge neutrality and thermal equilibrium at the boundary.
They are given by
\begin{align*}
\rho_D = \sqrt{C}, \quad
 V_D = - \log\big(\rho_D^2/\delta_c^2 \big) + U, \quad
 S_D = \log \big(\rho_D^2/\delta_c^2 \big) + U,
\end{align*}
where $U$ is the applied voltage at the contacts of a semiconductor, and $\delta_c$ denotes the scaled intrinsic carrier concentration.
\end{remark}

The enthalpy function $h(t)$ accounts for electron-electron interactions and is required to satisfy the following assumption (cf.~\cite{unterreiter}).

\begin{assumption}
\label{gammaQDD:assenthalpyasylimit}
Let $h \in \mathcal{C}^1(\rr_+;\rr)$ be strictly monotone increasing with the derivative $h'$ bounded away from zero, and $h(t) \in \mathcal O(t^{\frac{4}{3}})$ for $t \rightarrow \infty$. Furthermore, for any $t\in\rr_+$, let the function $r_t(\eta)=h(t+\eta)-h(\eta)-h'(t)\eta$ satisfy
\[
 |r_t(\eta)-r_t(\xi)| \le L\delta|\eta-\xi|,
\]
for any $\eta,\xi\in\rr$ with $|\eta|,|\xi|\le\delta$ for some constants $L>0$ and $\delta>0$.
\end{assumption}

\begin{example}
An enthalpy function that satisfies Assumption \ref{gammaQDD:assenthalpyasylimit} is given by
\[
 h(t) := \begin{cases} \log(t), & t \leq t_0,\\ g(t), & t \geq t_0, \end{cases}\quad \text{ for } t_0,
\]
where $g$ satisfies the above assumption. Furthermore, the interpolating inequalities $g(t_0)=\log(t_0)$ and $g'(t_0)t_0=1$ need to be satisfied. 
\end{example}

\begin{remark}
The expression $\log(t)$ is often used as an interaction term for low densities. However, an asymptotic expansion of exchange-correlation terms based on Fermi--Dirac statistics yields (cf.~\cite{AbdUnt,JueAsymptotic})
\[
h(t)= \mathcal{O}(t^{\frac{4}{3}}) \quad\text{for}\quad t\to\infty.
\]
This  expansion basically means that the interaction term grows with a certain speed, which is faster than the logarithmic expression for low densities. This property is essential in Section~\ref{sec:gammasemiclassicallimit} when using a Stampaccia argument to derive the uniform boundedness in $\e>0$ for the potential $V$.
\end{remark}

In the sequel we will use the following shorthand notations
\begin{align*}
H^1_0(\Omega \cup \Gamma_N) &:= \overline{ C^{\infty}_c(\Omega \cup \Gamma_N) }^{\| \. \|_{H^1(\Omega)} },\\
\Y_0 & := H^1_0(\Omega \cup \Gamma_N) \cap L^{\infty}(\Omega) , \qquad \| \. \|_{\Y_0}  := \| \. \|_{H^1(\Omega)} + \| \. \|_{L^{\infty}(\Omega)}, \\
\Y_1 & := \rho_D + \Y_0, \quad \Y_2 := V_D + \Y_0, \quad \Y_3 := S_D + \Y_0, \\
\Y & :=  \Y_1\times \Y_2\times \Y_3,
\end{align*}
and the admissible set of controls
\[
 \mathcal{C} := \{ C \in H^1(\Omega)\;|\; C = C_{\text{ref}} \; \text{on}\; \Gamma_D \},
\]
with a given reference doping profile $C_\text{ref}\in H^1(\Omega)\cap L^\infty(\Omega)$.
\subsection{Prelimenary results}
It is convenient to write system (\ref{QDD}) as a single equation with the help of solution operators. 

To incorporate the inhomogeneous Poisson equation (\ref{eqV}), we define the solution operator $\Phi: L^2(\Omega) \mapsto \Y_0;\,f\mapsto\Phi[f]$ as the unique weak solution of
\[
-\lambda^2  \Delta \Phi  = f \quad \text{in } \Omega, \quad 
\Phi = 0 \quad\text{on } \Gamma_D, \quad
\partial_\nu\Phi = 0 \quad\text{on } \Gamma_N,
\]
and $\Phi_e \in \Y_2$ as the unique weak solution of
\[
-\lambda^2 \Delta \Phi_e = 0 \quad \text{in } \Omega, \quad
\Phi_e = V_D \quad\text{on } \Gamma_D, \quad 
\partial_\nu \Phi_e = 0 \quad \text{on } \Gamma_N.
\]
Using the superposition principle, we can write $V = \Phi_V[\rho^2 - C] := \Phi[\rho^2 - C] + \Phi_e$, with $\Phi_V\colon L^2(\Omega)\to \Y_2$.

Likewise, the Fermi potential may be written as $S = \Phi_S[\rho^2]$ where the solution operator $\Phi_S: L^\infty(\Omega) \rightarrow \Y_3$ is defined as the unique weak solution of
\[
-\diver(\rho^2 \n \Phi_S) = 0 \quad \text{in } \Omega,\quad
\Phi_S = S_D \quad \text{on } \Gamma_D,\quad
\rho^2\partial_\nu \Phi_S = 0 \quad \text{on } \Gamma_N,
\]
for $\rho \geq \underline{\rho}$ almost everywhere in $\Omega$ for some constant $\underline{\rho}>0$.

The solution operators $\Phi, \Phi_e$ and $\Phi_S$ are well defined due to standard elliptic theory \cite{gil}. Consequently, we may write the weak solutions of system (\ref{QDD}) as solutions of the variational equation
\begin{align}\label{weakformulation1}
 \rho\in \Y_1:\quad \langle e_{\e}(\rho,C), \test \rangle = 0\quad\quad \forall\,\test\in H_0^1(\Omega\cup\Gamma_N).
\end{align}
where the nonlinear operator $e_{\e}\colon \Y_1 \times \mathcal{C} \rightarrow H_0^1(\Omega\cup\Gamma_N)^*$ is given by
\[
\langle e_{\e}(\rho,C), \test \rangle := -\e^2 \int \nabla \rho \.\nabla \test\, dx+ \int \rho \big( h(\rho^2) + \Phi_V[\rho^2 - C] - \Phi_S[\rho^2]\big)\test\,dx 
\]
We recall existence results and a priori estimates to (\ref{QDD})  shown in \cite{AbdUnt, HinPin01, UntVol}:

\begin{proposition}
\label{existenceweaksol}
Let Assumption \ref{assdomainGAMMA}  be satisfied. Then, for any $C \in \mathcal{C}$ and every data $(\rho_D,V_D,S_D)$ with
\[
 1/K\le \rho_D \le K\quad\text{a.e.~in}\;\;\Omega, \quad \|V_D\|_{L^\infty(\Omega)}, \|S_D\|_{L^\infty(\Omega)}\le K
\]
for some $K\ge 1$, there exists $(\rho,V, S) \in \Y$ with $V := \Phi_V[\rho^2-C]$ and $S := \Phi_S[\rho^2]$ satisfying (\ref{weakformulation1}), and the estimates
\begin{align}\label{aprioriestimate}
 \rho \ge 1/L\quad\text{a.e.~in}\;\; \Omega,\quad \| \rho\|_{\Y_1} + \| V \|_{\Y_2} + \| S \|_{\Y_3} \leq L,
\end{align}
for some constant $L\ge 1$ depending only on $\Omega$, $K$, and $\|C\|_{L^p(\Omega)}$, where the compact embedding $H^1(\Omega)\hookrightarrow L^p(\Omega)$ holds true.
\end{proposition}


\begin{remark}
Solutions of system (\ref{QDD}) may be related to the extrema of the quantum energy defined by the functional 
\begin{align}
\label{gammaQDDenergy}
E^{\e}_{S}(\rho) &:= \e^2 \int_{\Omega} | \n \rho|^2\, dx + \int_{\Omega} H(\rho^2)\, dx + \frac{\lambda^2}{2} \int_{\Omega} \left|\n V\right|^2dx - \int_{\Omega} S \rho^2\, dx
\end{align}
where $H(t) := \int_1^t h(s)\, ds$ and $C \in \mathcal{C}$, $S \in \Y_3$ are given. If $(\rho,V, S) \in \Y$ is the solution from Proposition \ref{existenceweaksol}, then $\rho$ is a minimizer of $E^{\e}_S$ in $\Y_1$. Furthermore, for a given $S$, $\rho$ is the unique minimizer of $E^{\e}_S$ (cf.~\cite{AbdUnt, UntVol}). 
\end{remark}

\begin{remark}
One drawback is that the solution from Proposition~\ref{existenceweaksol} need not be unique. Uniqueness may be assured for small applied voltages $U$ (see~Proposition~\ref{existencesolunique} below) and is essential when formulating an optimization algorithm. However, uniqueness is not required for the purpose of this paper.
\end{remark}

\section{Optimization Problem}
\label{sec:optimization}

We state the optimization problem with a general cost functional $J$ satisfying the following assumptions.
 
\begin{assumption}
\label{assumptioncostgammaQDD}
 Let $J\colon H^1(\Omega) \times \mathcal{C}\to\rr$ denote a cost functional, which is continuously Fr\'echet differentiable with Lipschitz continuous derivatives,  bounded from below, and radially unbounded with respect to the second variable. Furthermore, let $J$ be of separated type, i.e., we can write $J(\rho,C) = J_a(\rho) + J_b(C)$ with $J_a\colon H^1(\Omega) \mapsto \rr$ being weakly continuous and $J_b$ being weakly lower semicontinuous in $H^1(\Omega)$.
 \end{assumption}

\begin{example}
For example, the cost functional 
\begin{align}
\label{cost1}
J_1(\rho,C) := \frac{1}{2} \| \rho^2 - n_d \|_{L^2(\Omega)}^2 + \frac{\gamma}{2} \| \nabla(C - C_{\text{ref}}) \|_{L^2(\Omega)}^2,\quad\gamma>0,
\end{align}
satisfies Assumption~\ref{assumptioncostgammaQDD} due to the compact Sobolev embedding $H^1(\Omega) \hookrightarrow L^p(\Omega)$, $p\in[1,6)$ up to $d=3$. Therefore the tracking type term  is continuous with respect to the weak topology in $H^1(\Omega)$, and the second term is weakly lower semicontinuous since the norm is weakly lower semicontinuous (cf.~\cite{HinPin})
\end{example}

\begin{example}
Another cost functional often used for the optimal design of a semiconductor by tracking the total current on the boundary is (cf.~\cite{BurPin})
\begin{align}
\label{cost2}
J_2(\rho,C) := \frac{1}{2} | I(\rho) - I_d |^2 + \frac{\gamma}{2} \| \nabla(C - C_{\text{ref}}) \|_{L^2(\Omega)}^2,
\end{align}
with total current $I(\rho)$ given by
\[
 I(\rho) := \int_{\Gamma_O} \rho^2 \partial_\nu\Phi_S[\rho^2]\,ds,
\]
where $\nu $ is the outer normal and $\Gamma_O \subset \Gamma_D$. Unfortunately, this functional does not satisfy Assumption~\ref{assumptioncostgammaQDD} due to the boundary integral and the lack of an appropriate embedding for the trace operator.
\end{example}

Now we can formulate the optimization problem:

\begin{problem}\label{optgammaQDD}
Let $J$ satisfy Assumption~\ref{assumptioncostgammaQDD}.
The optimal control problem reads: Find $(\rho_*,C_*) \in \Y_1 \times \mathcal{C}$ such that
\[
J(\rho_*, C_*) = \min_{(\rho,C) \in \Y_1 \times \mathcal{C} } J(\rho,C) \quad
\text{s.t.}\quad e_{\e}(\rho,C) = 0  \quad \text{in}\;\; H_0^1(\Omega\cup\Gamma_N)^*.
\]
\end{problem}

The existence of a minimizer for any cost functional satisfying Assumption~\ref{assumptioncostgammaQDD} is a consequence of the existence results and a priori bounds of Proposition $\ref{existenceweaksol}$. Arguments used in the proof mimic those in \cite{UntVol}. 

\begin{theorem}\label{existenceoptcont}
There exists at least one solution $(\rho_*,C_*) \in \Y_1 \times \mathcal{C}$ to Problem~\ref{optgammaQDD}.
\end{theorem}

\begin{proof}
Since $J$ is bounded from below we can define 
\[
 j:= \inf_{(\rho,C)\in\Y_1 \times \mathcal{C}} J(\rho,C) > -\infty.
\] 
Now we choose a minimizing sequence $(\rho_k,C_k)$. From the radial unboundedness of $J$ with respect to $C$ we get a uniform bound for $(C_k)\subset H^1(\Omega)$. 
Hence, the a priori estimate (\ref{aprioriestimate}) yields the uniform boundedness of $(\rho_k)\subset\Y_1$. We can therefore deduce the existence of a subsequence, again denoted by $(\rho_k,C_k)$, and a $(\rho_*,C_*) \in \Y_1 \times \mathcal{C}$ such that
\[
\rho_k \rightharpoonup \rho_* \quad \text{in}\;\; H^1(\Omega),\quad \rho_k \rightharpoonup^* \rho_* \quad \text{in}\;\; L^\infty(\Omega),\quad C_k \rightharpoonup C_* \quad \text{in}\;\; H^1(\Omega).
\]
These convergences are sufficient to pass to the limit in the weak formulation. We begin with the term $\rho^2$. From the weak convergence in $H^1(\Omega)$ we deduce for $p<6$ the strong convergence of a subsequence of $(\rho_k)$ in $L^p(\Omega)$, and consequently yet another subsequence (denoted again by $(\rho_k)$) that converges a.e.~in $\Omega$, i.e.,
\[
 \rho_k(x)\to \rho_*(x)\quad\text{a.e.~in}\;\;\Omega.
\]
Therefore, passage to the limit in the nonlinear terms $\rho^2$ and $\rho h(\rho^2)$ may be shown by a simple application of the Lebesgue dominated convergence.
Also in the other terms, the above convergences are sufficient to pass to the limit. Indeed, by defining $V_k := \Phi_V[\rho_k^2 - C_k]$, and $S_k := \Phi_S[\rho_k^2]$, we obtain from estimate (\ref{aprioriestimate})
\[
S_k \rightharpoonup S_* \quad \text{in}\;\; H^1(\Omega),\quad V_k \rightharpoonup V_* \quad \text{in}\;\; H^1(\Omega)
\]
for some $ V_* \in \Y_2$ and $S_* \in \Y_3$. The continuity and linearity of $\Phi_V$ allows to identify $V_*=\Phi_V[\rho_*^2-C]$. As for $S_*$, we use the Lebesgue dominated convergence again to recover the strong convergence of $\rho_k^2\n \test \to \rho_*^2\n \test$ in $L^2(\Omega)$ for any $\test\in H^1(\Omega)$, and consequently the convergence
\[
 (\rho_k^2\n S_k,\n \test) = (\n S_k,\rho_k^2\n \test) \to (\n S_*,\rho_*^2\n \test)\quad\text{for}\;\;k\to\infty,
\]
for all $\test\in H^1(\Omega)$. This implies that $S_*=\Phi_S[\rho_*^2]$. Altogether, we obtain
\[
0 = \langle e_{\e}(\rho_k,C_k), \test \rangle \to \langle e_{\e}(\rho_*,C_*), \test \rangle \quad\text{for all } \test\in H_0^1(\Omega\cup \Gamma_N).
\]
By standard arguments due to the weak lower semicontinuity of $J$, we finally conclude that $(\rho_*,C_*)\in \Y_1\times \mathcal{C}$ solves the minimization problem (\ref{optgammaQDD}).
\end{proof}

\begin{remark}
 As a result of Theorem~\ref{existenceoptcont}, we obtain, for every $\e>0$, the existence of at least one triple $(\rho_\e,V_\e,S_\e)\in \Y$, which resolves Problem~\ref{optgammaQDD}.
\end{remark}

\section{$\Gamma$--Convergence for the Semi-classical Limit}
\label{sec:gammasemiclassicallimit}

As mentioned earlier, we will use the concept of $\Gamma$--convergence to prove the convergence of minima and minimizers in the semi-classical limit. An introduction into this topic may be found, e.g., in \cite{Braides,DalMas}. In essence, $\Gamma$--convergence of functionals can be characterized by the following two inequalities, the so-called liminf-inequality and limsup-inequality, as given in the following proposition.

\begin{proposition}[$\Gamma$--convergence of functionals]
\label{defgammaconvsequ}
Let $X$ be a reflexive Banach space with a separable dual and $(F_k)$ be a sequence of functionals from $X$ into $\overline{\rr}$. Then $(F_k)$ $\Gamma$--converges to $F$ if the following two conditions are satisfied:
\begin{enumerate}
\item[\em (i)] For every $x \in X$ and for every sequence $(x_n) \subset X$ weakly converging to $x$:
\begin{align}
\label{liminfinequ}
F(x) \leq \liminf_{k \rightarrow \infty} F_k(x_k).\tag{L-inf}
\end{align}
\item[\em (ii)] For every $x \in X$ there exists a sequence $(x_n)\subset X$ weakly converging to $x$:
\begin{align}
\label{limsupequ}
F(x) \geq \limsup_{k \rightarrow \infty} F_k(x_k).\tag{L-sup}
\end{align}
\end{enumerate}
\end{proposition}

We will also need the notion of {\em weak equi-coercivity} for functionals on reflexive Banach spaces, endowed with its weak topology.

\begin{definition}\label{equicoercive}
Let $X$ be reflexive Banach space with a separable dual. A sequence $(F_k)$ is said to be {\em weakly equi-coercive} on X, if for every $t \in \rr$ there exists a weakly compact subset $K_t$ of $X$ such that $\{ F_k \leq t \} \subseteq K_t$ for every $k \in \nn$.
\end{definition}					

Altogether, it holds \cite{Braides}:
\[
\Gamma\text{\em --convergence}  + \text{\em equi-coercivity} \;\Rightarrow \;\text{\em convergence of minima and minimizers.} 
\]
We apply this concept to our semi-classical limit problem. Notice that $X$ as a product of reflexive Banach spaces with separable duals is again a reflexive Banach space with a separable dual. Moreover, for each $\e >0$, we define $\xe$ as the set of all admissible pairs,
\[
\xe := \left\{(\rho,C) \in \Y_1 \times \mathcal{C}\;|\; e_{\e}(\rho,C) = 0\;\;\text{in}\;\;H_0^1(\Omega\cup\Gamma_N)^* \right\},
\]
and let $\chi_{\e}\colon X\to\overline{\rr}$, be its characteristic function, given by
\[
\chi_{\e}(\rho,C) =
\begin{cases}
0 &\text{if } (\rho,C) \in \xe, \\
+\infty &\text{otherwise}.
\end{cases}
\]
To define the set of admissible pairs $\xnull$ for classical solutions, we first make the following assumption:
\begin{assumption}\label{gammaQDD:assclassicalsolutions}
If $(\rho_0,C_0) \in \Y_1 \times \mathcal{C}$ satisfies $e_0(\rho_0,C_0) = 0$, then the solution $(\rho_0,V[\rho_0],S[\rho_0])$ of system (\ref{QDD}) is isolated. 
\end{assumption}

\begin{remark}
 Since we do not assume uniqueness of solutions for the classical drift-diffusion equation, Assumption~\ref{gammaQDD:assclassicalsolutions} is necessary to provide for the invertibility of the linearisation at the point $\rho_0$, which is essential in the proof of Lemma~\ref{QDD:lemmaexistenceconvergentsolutions}.
\end{remark}

In the following, we restrict the classical solution space to
\[
\xnull := \left\{(\rho,C) \in \Y_1 \times \mathcal{C}\;|\; (\rho,C)\;\;\text{satisfies Assumption}~\ref{gammaQDD:assclassicalsolutions} \right\} .
\]

\begin{remark}
Note that $\xnull \neq \emptyset$ due to existence results in \cite{AbdUnt}.
Moreover, the classical analogue to the quantum energy (\ref{gammaQDDenergy}) is given as
\begin{align}
\label{QDD:energyDD}
E^0_S(\rho)  := \int_{\Omega} H(\rho^2) \, dx + \frac{\lambda^2}{2} \int_{\Omega} | \nabla V|^2\, dx  - \int_{\Omega} S\rho^2\, dx,
\end{align}
for given $S \in \Y_3$. It is well-known, that if $(\rho,V,S) \in \Y$ solves the drift-diffusion equations, then $\rho$ is the unique minimizer of $E^0_S$ in $\Y_1$ (cf.~\cite{AbdUnt}).
\end{remark}

Thus, the task at hand is to prove the following theorem:
\begin{theorem}\label{charfunction:gammaconv}
Let $\xe$ and $\xnull$ be as defined above, and let $(\e_k)$ be a sequence with $\e_k \rightarrow 0$ as $k \rightarrow \infty$. Then $(\chi_{\e_k})$ $\Gamma$--converges to $\chi_{0}$.
\end{theorem}

\begin{remark}
 Theorem~\ref{charfunction:gammaconv} may also be equivalently interpreted as the {\em Kuratowski convergence} of the sets $\xe$ towards $\xnull$ (cf.~\cite{DalMas}).
\end{remark}

For the proof of Theorem~\ref{charfunction:gammaconv}, we will need the following two lemmata, whose proofs are found in Appendix~\ref{proof:lemmawbound} and \ref{proof:QDD:lemmaexistenceconvergentsolutions}, respectively.

\begin{lemma}
\label{lemmawbound}
Let $(\e_k)$ be a sequence with $\e_k \rightarrow 0$ for $k \rightarrow \infty$ and $(x_k)$ be a sequence with $x_k=(\rho_k,C_k) \in \Xi_{\e_k}$ for all $k \in \nn$. If the sequence $(C_k)$ is bounded in $H^1(\Omega)$, then the sequences $(\rho_k)$, $(V_k) = (\Phi_V[\rho_k^2 - C_k])$ and $(S_k) =(\Phi_S[\rho_k^2])$ are uniformly bounded, i.e., there exists a constant $M > 0$ independent of $k$ such that 
\begin{align}
\label{uniformH1boundw}
\| \rho_k \|_{\Y_1} + \| V_k \|_{\Y_2} + \| S_k \|_{\Y_3} \leq M \quad \text{for all}\;\; k \in \nn.
\end{align}
Furthermore, there exist uniform lower and upper bounds $0 < \underline{\rho} < \overline{\rho}$ with
\begin{eqnarray}
\label{QDD:uniformboundsw}
\underline{\rho} \leq \rho_k \leq \overline{\rho} \quad \text{a.e.~in}\;\;\Omega,\ \text{for all}\;\; k \in \nn.
\end{eqnarray}
\end{lemma}

As a result of Lemma~\ref{lemmawbound}, one may extract subsequences in $(\e_k)$ that converge to some weak limit in $X$, and hope to classify the limit as a solution of the classical drift-diffusion equations. A similar result for a fixed doping profile $C$ was shown in \cite{AbdUnt}. 

\begin{lemma}
\label{QDD:lemmaexistenceconvergentsolutions}
Let $(\rho_0,C) \in \xnull$ satisfy Assumption~\ref{gammaQDD:assclassicalsolutions}. Then there exists a sequence $(\e_k)$ with $\e_k \rightarrow 0$ as $k \rightarrow \infty$ and a sequence $(\rho_k)$ with $(\rho_k, C) \in \Xi_{\e_k}$ for all $k\in\nn$ such that $\rho_k \rightharpoonup \rho_0$ in $H^1(\Omega)$.
\end{lemma}

The idea of the proof of Lemma~\ref{QDD:lemmaexistenceconvergentsolutions} lies in the fact that the quantum model is a regular perturbations of the classical model for sufficiently small $\e>0$. Therefore, standard tools from asymptotic analysis and the implicit function theorem allow to construct solutions to the quantum model, with the help of solutions satisfying Assumption~\ref{gammaQDD:assclassicalsolutions}. For completeness, we include the technical details in Appendix~\ref{proof:QDD:lemmaexistenceconvergentsolutions}.

\begin{proof}[Proof of Theorem~\ref{charfunction:gammaconv}]
We need to show (\ref{liminfinequ}) and (\ref{limsupequ}) of Proposition~\ref{defgammaconvsequ}.

We begin with (\ref{liminfinequ}): Let $x=(\rho,C) \in \xnull$, i.e., $\chi_{0}(x) = 0$. Since the characteristic function $\chi_{0}$ only takes the values $0$ and $+\infty$, the inequality is satisfied trivially. Now suppose $x \notin \xnull$, i.e., $\chi_{0}(x) = +\infty$. Let $(x_k)$ be a sequence converging weakly to $x$ in $X$. Suppose otherwise, i.e., $\liminf_{k \rightarrow \infty} \chi_{\e_k}(x_k) = 0$. Consequently, there exists a subsequence, again denoted by $(x_k)$, with $x_k \in {\Xi_{\e_k}}$ for all $k \in \nn$. From the weak convergence we obtain the boundedness of the sequence in $X$, i.e.,
\[
\| \rho_k \|_{H^1(\Omega)} + \| C_k \|_{H^1(\Omega)} \leq c \quad \text{for all}\;\; k  \in \nn,
\]
for some $c>0$. Since $\rho_k$ is uniformly bounded in $H^1(\Omega)$, we may pass to the limit in the first term on the right hand side of (\ref{weakformulation1}) to obtain
\[
 \e_k^2 \big( \nabla \rho_k, \nabla \test \big)_{L^2(\Omega)}\to 0\quad\text{for}\;\;k\to\infty.
\]
Furthermore, we infer from Lemma~\ref{lemmawbound} the existence of yet another subsequence, denoted again by $(x_k)$, that converges weakly to $x_*=(\rho_*,C_*) \in \Y_1 \times \mathcal{C}$ satisfying
\begin{align*}
\rho_k \rightharpoonup \rho_* \quad \text{in}\;\; H^1(\Omega),& \quad  C_k \rightharpoonup C_* \quad \text{in}\;\; H^1(\Omega),\\
\rho_k(x) \to \rho_*(x) \quad \text{a.e.~in}\;\; \Omega,& \quad \rho_k \rightharpoonup^* \rho_* \quad \text{in}\;\; L^\infty(\Omega).
\end{align*}
As for the remaining terms in (\ref{weakformulation1}) we can proceed as in the proof of Theorem~\ref{existenceoptcont} to conclude the convergence
\[
\langle e_0(x_*), \test \rangle = 0 \quad\text{for all}\;\;\test\in H_0^1(\Omega\cup\Gamma_N) \quad\Longrightarrow \quad x_*  \in \xnull.
\]
which contradicts our assumption $x \notin \xnull$, simply due to the uniqueness of weak limits. Therefore, (\ref{liminfinequ}) must hold true.

We now show (\ref{limsupequ}): Let $x \notin \xnull$, i.e., $\chi_{0}(x)=+\infty$. Then (\ref{limsupequ}) is trivially satisfied for any sequence because $\chi_{\e_k}$ only takes either the value $0$ or $+\infty$. Now let $x \in \xnull$, i.e., $\chi_{0}(x)=0$. Due to Lemma \ref{QDD:lemmaexistenceconvergentsolutions} there exists a sequence $(x_k)$ with $x_k=(\rho_k,C)$ converging weakly to $x$ in $X$ with $x_k \in \Xi_{\e_k}$, i.e., $\chi_{\e_k}(x_k) = 0$ for all $k \in \nn$. Hence, (\ref{limsupequ}) is also satisfied in this case. This concludes the proof.
\end{proof}

\section{Convergence of Minima and Minimizers}

To include the state equation into the cost functional we use the characteristic function and define the cost functional $\J_{\e}$ as
\begin{align}
\label{costfunctionalgamma}
 \J_{\e}  =  J + \chi_{\e},
\end{align}
where $J$ is a functional satisfying Assumption~$\ref{assumptioncostgammaQDD}$. Problem~\ref{optgammaQDD} can now be formulated as follows:

\begin{problem}\label{opt2}
Find $(\rho_*, C_*) \in \Y_1 \times \mathcal{C}$ such that
\[
 \J_{\e}(\rho_*, C_*) = \min_{(\rho,C) \in \Y_1 \times \mathcal{C}} \J_{\e}(\rho,C).
\]
\end{problem}

\begin{remark}
Note that $\J_{\e}$ is no longer Fr\'{e}chet differentiable. However, this is not crucial since Problem~\ref{optgammaQDD} and Problem~\ref{opt2} are equivalent. 
\end{remark}

To prove the $\Gamma$--convergence of $(\J_{\e})$ to $\J_{0}$, we use the fact that $(\chi_{\e})$ $\Gamma$--converges to $\chi_{0}$, along with Assumption \ref{assumptioncostgammaQDD} on the cost functional $J$. This is sufficient to prove the $\Gamma$--convergence of $(\J_{\e})$ to $\J_{0}$.

\begin{theorem}
Let $(\J_{\e})$ and $\J_{0}$ be defined as above. Then 
\[
{\Gamma\text{--}\lim}_{\e \rightarrow 0}\,\J_{\e} = \J_{0},
\]
i.e., $\J_{0}$ is the $\Gamma$--limit of $\J_{\e}$.
\end{theorem}
\begin{proof}
As in the proof of Theorem~\ref{charfunction:gammaconv}, let $(\e_k)$ be a zero sequence for $k \rightarrow \infty$, with
\[
 X = H^1(\Omega)\times H^1(\Omega).
\]

To see that the sequence $\left( \J_{\e_k} \right)$ satisfies (\ref{liminfinequ}), we note that $J$ is weakly lower semicontinuous due to Assumption \ref{assumptioncostgammaQDD}. Now let $x=(\rho,C) \in X$ and $(x_k)\in X$ be a sequence converging weakly to $x$, then we can estimate
\[
\J_{0}(x) = J(x) + \chi_{0}(x) \le \liminf_{k\rightarrow \infty} J(x_k) + \liminf_{k\rightarrow \infty} \chi_{\e_k}(x_k) \le \liminf_{k\rightarrow \infty} \J_{\e_k}(x_k),
\]
which gives the liminf-inequality.

To show (\ref{limsupequ}) we exploit the special structure of $J$ ensured by Assumption~\ref{assumptioncostgammaQDD}. 
Let $x=(\rho,C) \in X$. If $\J_{0}(x) = +\infty$, we define the constant sequence $x_k = x$ for all $k \in \nn$. For sufficiently large $N \in \nn$, we have $\J_{\e_k}(x_k) = +\infty$ for $k>N$, because otherwise (\ref{liminfinequ}) would be violated. Therefore (\ref{limsupequ}) holds for this case. If $\J_{0}(x) < \infty$, i.e., $x \in \xnull$, we can argue as in the proof of Theorem~\ref{charfunction:gammaconv} to show the existence of a sequence $(\rho_k,C)\in \Xi_{\e_k}$ converging weakly to $x$ in $X$. Assumption~\ref{assumptioncostgammaQDD} ensures the continuity of $J$ with respect to the weak topology. Since the doping profile $C$ is constant, this yields $J(x_k) \to J(x)$ as $k\to\infty$, and consequently also $\J_{\e_k}(x_k) \to \J_0(x)$ as $k\to\infty$. Thus, (\ref{limsupequ}) also holds in this case.
\end{proof}

It remains to show the weak equi-coercivity of the functionals $\J_{\e}$, which may be shown with the help of Lemma \ref{lemmawbound}.

\begin{theorem}
\label{corollaryequicoercive}
 The sequence $( \J_{\e})$ is weakly equi-coercive in $X$.
\end{theorem}
\begin{proof}
Recalling Definition \ref{equicoercive}, we show that the subset $\{ \J_{\e} \leq t \}$ is bounded w.r.t.~the strong topology in $X$ for every $t \in \rr$. In particular, the norm
\[
\| x \|_X =\| \rho \|_{H^1(\Omega)} + \| C \|_{H^1(\Omega)}
\]
must be bounded for every $x=(\rho,C)$ with $ \J_{\e}(x) \leq t$ for every $t \in \rr$.

Let $t < \infty$. Then every $(\rho,C) \in \{ \J_{\e} \leq t \}$  must be in the set of admissible states $\Xi_{\e}$ for some $\e >0 $ due to the characteristic function $\chi_{\e}$ in the definition of the cost functionals in (\ref{costfunctionalgamma}). Furthermore, $\|C\|_{H^1(\Omega)}$ must be uniformly bounded in $\e$ due to the radial unboundedness of $J$ w.r.t.~$C$ and because the term $J_b$ is independent of $\e$. Due to Lemma~\ref{lemmawbound} we also have the uniform boundedness of $\|\rho\|_{H^1(\Omega)}$, and therefore the weak equi-coercivity of the sequence $(\J_{\e})$.
\end{proof}

Now all the assumptions to show the convergence of minima are satisfied, which we recall in the following proposition (cf.~\cite[Theorem~7.8]{DalMas}).

\begin{proposition}[Convergence of minima]\label{convminima}
Let $(\J_{\e})$ and $\J_{0}$ be defined as above. Then $\J_{0}$ attains its minimum on $X$ and
\[
\min_{x \in X} \J_{0}(x) = \lim_{\e \rightarrow 0} \min_{x \in X} \J_{\e}(x).
\]
\end{proposition}

The convergence of minima allows us to also show the convergence of minimizers.

\begin{corollary} [Convergence of minimizers]
\label{convminimizers}
Let $\J_{\e}$ and $\J_{0}$ be defined as above, and let $(\e_k)$ be a sequence with $\e_k \rightarrow 0$ as $k \rightarrow \infty$ and $x_k^* = ( \rho_k^*, C_k^*)$ such that
\[
\J_{\e_k}(x_k^*)  = \min_{x \in X} \J_{\e_k}(x).
\]
Then, there exists a subsequence, again denoted by $(x_k^*)$, such that 
\[
 x_k^* \rightharpoonup x_{0}^* \quad \text{in}\;\; X
\]
with $x_{0}^* \in \xnull$ and 
\[
\J_{0}( x_{0}^*) = \min_{x \in X} \J_{0}(x),
\]
i.e., $x_0^*$ is a minimizer of $\J_{0}$.
\end{corollary}

\begin{proof}
Due to Proposition~\ref{convminima}, there exists $N\in \nn$ and some constant $c>0$ such that $\J_{\e_k}(x_k^*) < c$ for all $k\geq N$. Due to the weak equi-coercivity from Theorem~\ref{corollaryequicoercive}, the sequence $(x_k^*)$ contains a subsequence, again denoted by $(x_k^*)$, which converges weakly to some $x_0^*$ in $X$. With the same arguments as in the proof of Theorem~\ref{charfunction:gammaconv}, we conclude that $x_0^* \in \xnull$. The assertion follows from \cite[Corollary~7.20]{DalMas}.
\end{proof}

\section{Numerical Results}
\label{gammaQDDnumresults}

In this section we give a numerical example for the convergence of minima and minimizers. For the simulation, we solve the system
\begin{subequations}
\label{eq:QEP}
\begin{align}
\label{QEP2}
-\e^2 \Delta \rho + \rho(h(\rho^2) + V +V_{\text{ext}} - S)  &= 0, \\
\label{QEP4}
-\lambda^2 \Delta V &= \rho^2 - C, \\
\label{QEP3}
-\diver(\rho^2 \nabla S) &= 0.
\end{align}
\end{subequations}
As enthalpy function we use
\[
h(t) := \log(t), \quad \text{for}\;\; t\leq K,\;\; K >0,
\]
and assume that our simulation stays in the regime of 'low' densities, i.e., that $t < K$ always holds. This assumption is very common in semiconductor modelling (cf.~\cite{HinPin, UntVol, HiPi07}). Recall that the physical electron density is given by $n=\rho^2$.

\begin{figure}[htbp]
\centering
\includegraphics[scale=0.2]{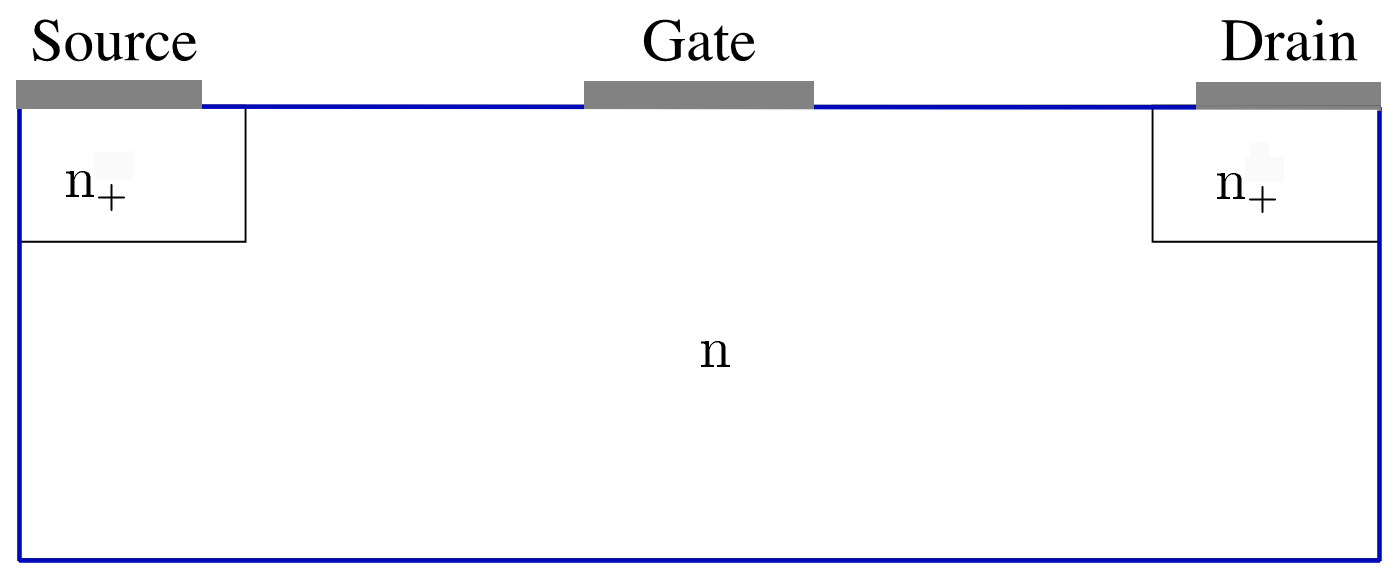}
\caption{Geometry of a MESFET.}
\label{fig:MESFET}
\end{figure}

We intend to optimize the current of a metal semiconductor field effect transistor (MESFET) device (cf.~\cite{Se84}) by adjusting the background doping profile. MESFETs are typically modelled by a rectangle $\Omega$ in two spatial dimensions (see Fig.\,\ref{fig:MESFET}), where the boundary $\Gamma:=\partial \Omega$ splits into two parts. Fig.\,\ref{fig:MESFET} is a simple sketch of a typical MESFET adapted from \cite{HolJuePie}. The grey parts (Ohmic contacts) represent the Dirichlet boundary $\Gamma_D$, where an external voltage may be applied. This consists of the source, drain and gate contacts, which controls the on/off state of the MESFET. We use the physically relevant boundary conditions described in Remark~\ref{boundarydata}. More specifically, the gate contact of the MESFET is modelled as a Schottky-contact by setting $\rho_D = \alpha_V \sqrt{C}$ for some $\alpha_V \in(0,1)$. Here, we choose $\alpha_V = 0.1$. The blue lines correspond to the insulating Neumann boundary $\Gamma_N$. The $n_+$ region denotes the highly doped part of the MESFET, while the $n$ region is considerably less doped than the $n_+$ region. 

For the well-posedness of the reduced cost functional
\[
\hat J(C) = J(\Phi_\rho[C], C),
\]
where $\Phi_\rho$ is the control-to-state map, which assigns for any $C\in \mathcal{C}$ a unique $\rho$ satisfying the equations (\ref{eq:QEP}), we require the uniqueness of solutions, which is ensured by the following result \cite{PinUnt}.
\begin{proposition}
\label{existencesolunique}
Let $C \in \mathcal{C}$. Then there is a constant $U_{\text{max}} > 0$ such that if
\[
\|U\|_{L^{\infty}(\Omega)} \le U_{\text{max}},
\]
the solution $(\rho,V, S) \in \Y$ from Proposition \ref{existenceweaksol} is unique.
\end{proposition}
In the following, we  assume that the applied voltage $U$ satisfies the assumption of Proposition~\ref{existencesolunique}. 
The voltage
\begin{align}
\label{QEP:externalPotential}
U = \alpha_V \cdot\underline{}\begin{cases}
0.15 & \text{at the drain}, \\
0.0375 & \text{at the source}, \\
0.075 & \text{at the gate}.
\end{cases}
\end{align}
proved to be sufficiently small for the forward simulation to work. 

The rest of the boundary, $\Gamma_N := \Gamma \setminus \Gamma_D$, represents the insulating parts of the boundary and is therefore of Neumann type, i.e., 
\[
 \partial_\nu \rho = \partial_\nu V = \rho^2\partial_\nu S= 0,
\]
where $\nu$ is the outward unit normal along $\Gamma_{N}$.

The aim of the optimization is to amplify the current $\rho^2\n S$ over the drain $\Gamma_{O} \subset \Gamma_D$ to reach a given value $I_d$. Since it is desirable to maintain the overall structure of the semiconductor device, large deviations from the initial doping profile should be penalized. Therefore, we define the cost functional $J$ as
\[
 J(\rho,S,C) = \frac{1}{2}\abs{I(\rho, S) - I_d}^2 + \frac{\gamma}{2}\|\n(C-C_{\text{ref}})\|_{L^2(\Omega)}^2
\]
with
\[
I(\rho,S) = \int_{\Gamma_{O}} \rho^2\partial_\nu S\,ds
\]
where $I_d\in \rr$ is the desired current flow on the drain $\Gamma_{O}$. Here, $C_{\text{ref}}$ is the reference doping profile (e.g.~the given MESFET), which is later also used as an initial guess for the optimization algorithm. The parameter $\gamma>0$ is a regularization parameter, which allows to adjust the deviations of the optimal profile from the reference $C_{\text{ref}}$. This type of cost functional is most commonly used in the design of optimal doping profiles \cite{BurPin,HinPin07,HinPin}.


\begin{remark}
The Lagrange multipliers $\xi=(\xi_\rho,\xi_V,\xi_S)$ corresponding to the first-order optimality condition of the optimization problem are required to formally solve the adjoint problem (see also \cite{UntVol}), given by the equations
\begin{subequations}
\begin{align*}
-\e^2\Delta \xi_\rho+\xi_\rho(2 + \log(\rho^2)+V-S)  \hspace*{5em}&\\
+\; 2\rho(\n S\.\n \xi_S - \lambda^2\xi_V) &= (I(\rho,S)-I_d)\n_S I(\rho,S) \\
-\lambda^2\Delta\xi_V + \rho\xi_\rho &= 0\\
-\diver(\rho^2\n\xi_S) - \rho\xi_\rho &= (I(\rho,S)-I_d)\n_\rho I(\rho,S)
\end{align*}
\end{subequations}
with homogeneous Dirichlet and Neumann boundary conditions on $\Gamma_D$ and $\Gamma_N$, respectively. Finally, the optimality condition reads
\begin{align}
\label{QEPoptcondition}
 \langle \hat J'(C_*),C\rangle = 0 \quad\text{for all}\;\; C\in \mathcal{C},
\end{align}
where
\[
 \langle \hat J'(C),\test\rangle = - \int_\Omega \xi_V\test + \gamma \n( C - C_d)\.\n \test\,dx\quad\text{for all }\,\test\in H^1_0(\Omega\cup\Gamma_N).
\]
For the update of the optimization algorithm, we will need the Riesz representative of the derivative denoted by $\n \hat J\in H_0^1(\Omega\cup\Gamma_N)$. This may be done by solving the Poisson problem:
\[
 \text{Find}\; g\in H^1_0(\Omega\cup\Gamma_N):\quad \int_\Omega \n g\.\n\test\,dx = \langle \hat J'(C),\test \rangle\quad\text{for all }\,\test\in H_0^1(\Omega\cup\Gamma_N).
\]
Note, that $g|_{\Gamma_D} = 0$, therefore, if we start with some function $C_0$ with $C_0|_\Gamma=C_{\text{ref}}|_\Gamma$. Then, the update will leave the values on the Dirichlet boundary unchanged.
\end{remark}


For the optimization we choose a gradient algorithm with Armijo-type line search as stated in \cite{HinPinUlbUlb}. 

\begin{algorithm}[h]
  \caption{Gradient Algorithm}
  \label{alg:GradientOptimization}
  \begin{tabular}{l p{\textwidth}}
    \algorithmicrequire \ A feasible doping profile $C_\text{ref}\in \mathcal{C}$, a tolerance ${\tt TOL} > 0$.\\
    \algorithmicensure \ A feasible doping profile $C_*$ with (locally) minimal costs.
  \end{tabular}
  \begin{algorithmic}[1]
    \STATE Set $k = 0$
	\STATE Set $C_k = C_\text{ref}$
	\STATE Compute gradient $g_k=\n \hat J(C_k)$.
    \WHILE{$\|g_k\|_{H^1(\Omega)}>{\tt TOL}$}
    \STATE $\alpha_k = \text{Armijo linesearch}_{\alpha > 0} \left\{\hat J\big(C_k - \alpha g_k\big)\right\}$
    \STATE $C_{k+1} = C_k - \alpha_k g_k$
    \STATE Set $k = k+1$
    \STATE Compute gradient $g_k=\n \hat J(C_k)$.
    \ENDWHILE
    \RETURN $C_* = C_k$
  \end{algorithmic}
\end{algorithm}

Consider the MESFET profile:
\begin{align*}
C_0(x) = \begin{cases}
1 & \text{if } x \text{ is in a highly doped $n_+$ region}\\
0.01 & \text{else}
\end{cases}
\end{align*}

\begin{remark}
The MESFET profile $C_0$ with jumps is  commonly used in semiconductor modelling. However, it holds that $C_0 \notin \mathcal{C}$, since $C_0 \notin H^1(\Omega)$. For this reason and also for numerical purposes, the reference profile $C_\text{ref}$ is chosen as a smoothed versions of $C_0$ (cf.~\cite{Schn11}) 
\end{remark}

We desire an amplification of the current by $100\%$, i.e., $I_d := 2 I(\rho_\text{ref},S_\text{ref})$, where $\rho_\text{ref}$ and $S_\text{ref}$ are computed with the reference doping profile $C_\text{ref}$.
All calculations are made on a grid with $80\times 80$ nodes and an error tolerance of $10^{-8}$. We choose the parameters $\lambda^2 = 0.0017$ and $\e^2 = 1.88\cdot 10^{-4}$. In \cite{Schn11} the grid independence of Algorithm \ref{alg:GradientOptimization} was also shown.

Optimization results for the quantum and classical drift-diffusion model are shown in Fig.\,\ref{fig:gammaQDDopt}. The relative deviation of the optimized doping profile to the reference profile for the quantum model can be seen in Fig.\,\ref{fig:gammaQDDreldeviationDopingProfileQDDReference}. We observe that the doping profile hardly changes in the highly doped regions and near to the contact while it is increased up to $300 \%$ in the channel. Fig.\,\ref{fig:gammaQDDreldeviationDensityQDDReference} shows that this changes the electron densities correspondingly, i.e., it varies only slightly in the upper part while it is increased up to $70 \%$ in the lower part. There is only a very small difference between the optimized profiles in the quantum and classical drift-diffusion model in Fig.\,\ref{fig:gammaQDDreldeviationDopingProfileQDDDD}. The same holds true for the electron density seen in Fig.\,\ref{fig:gammaQDDreldeviationDensityQDDDD}, except close to the gate contact. In this region there occur large gradients, which are detected in the quantum model (due to the Bohm potential $\e^2 \frac{ \Delta \rho}{\rho}$) but not in the classical model.

\begin{figure}[htbp]
\centering 
   	\subfigure[Reference doping profile.]{
   	\includegraphics[scale=0.15]{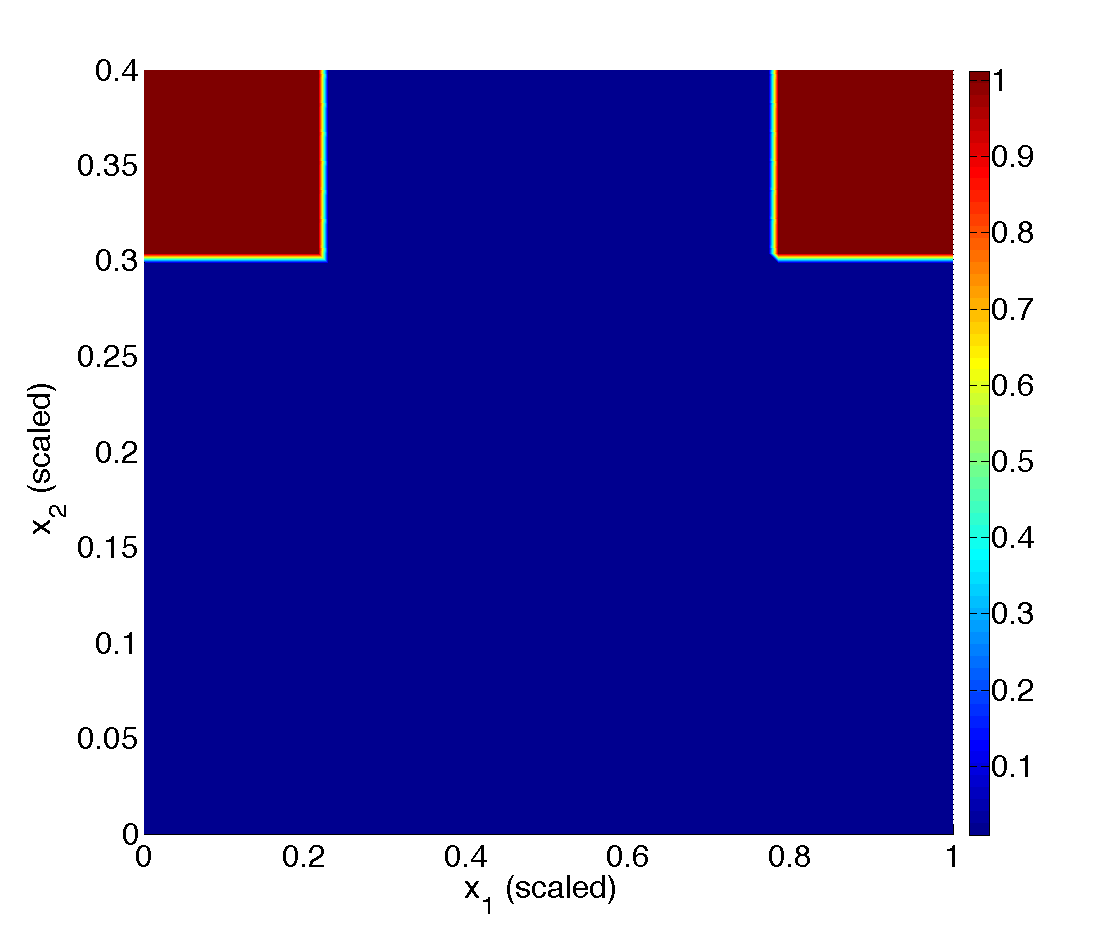}
   	\label{fig:DopingReferencegammaQDD}
   	}
   	\hspace*{0.3em}
    \subfigure[Total current density for reference doping profile of QDD ($\|n\n S\|_\infty = 0.557$).]{
    \includegraphics[scale=0.15]{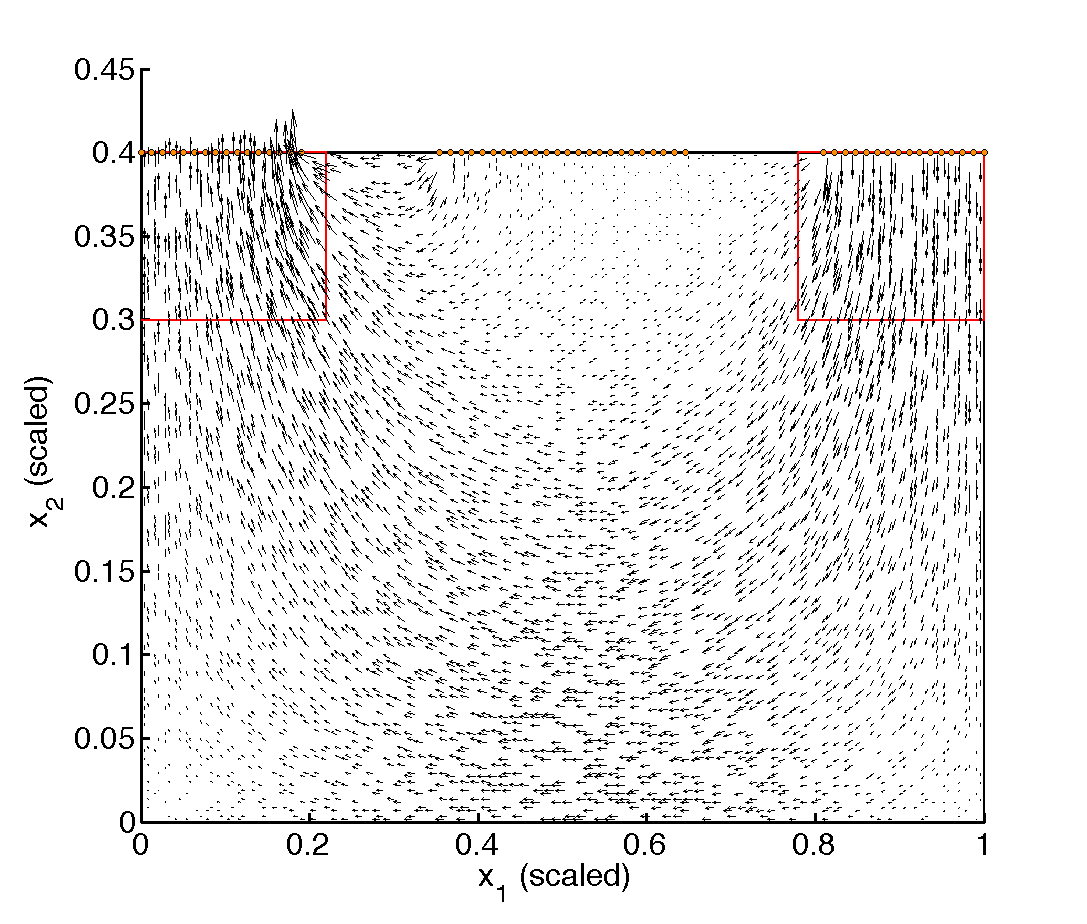}
    \label{fig:CurrentReferencegammaQDD}
    }
    
    \subfigure[Optimized doping profile of QDD.]{
    \includegraphics[scale=0.15]{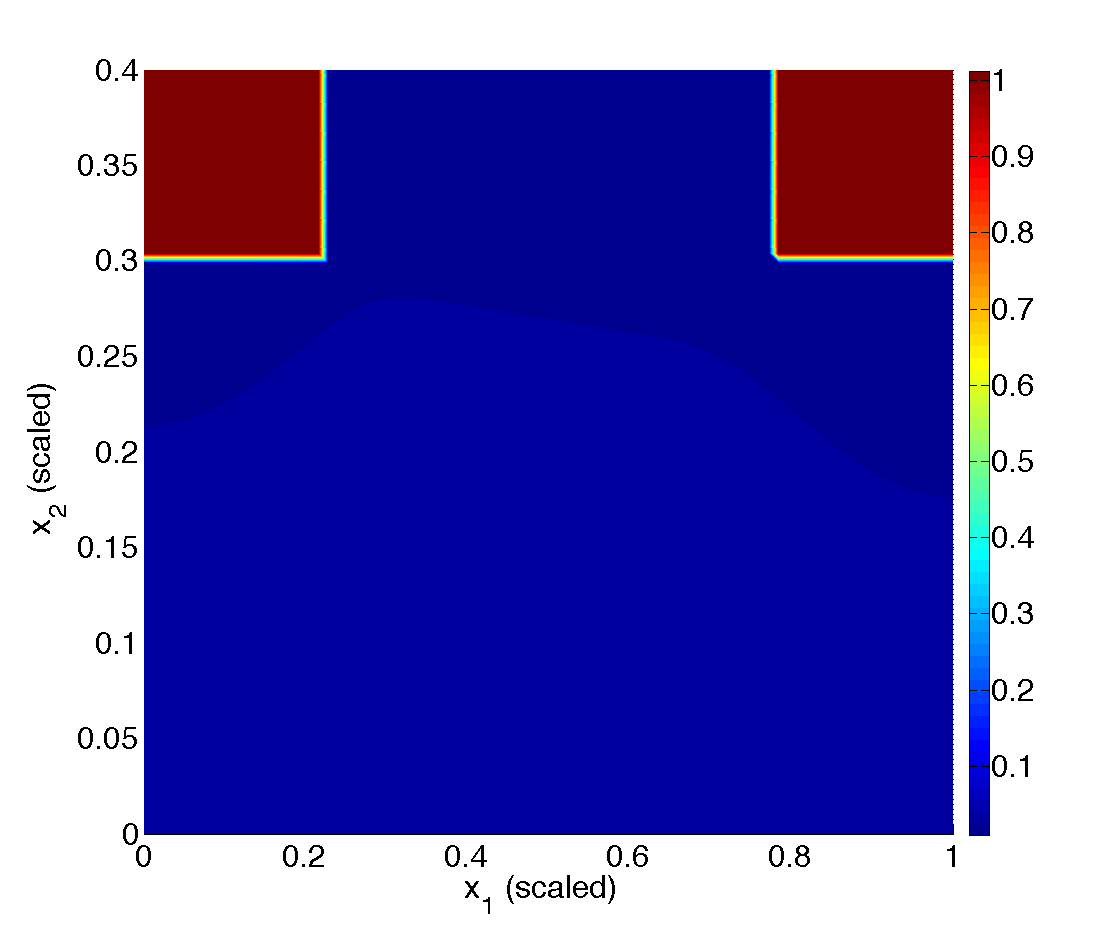}
    \label{fig:DopingoptQDDgammaQDD}
    }
    \hspace*{0.3em}
    \subfigure[Total current density for optimized doping profile of QDD ($\|n\n S\|_\infty = 0.741$).]{
    \includegraphics[scale=0.15]{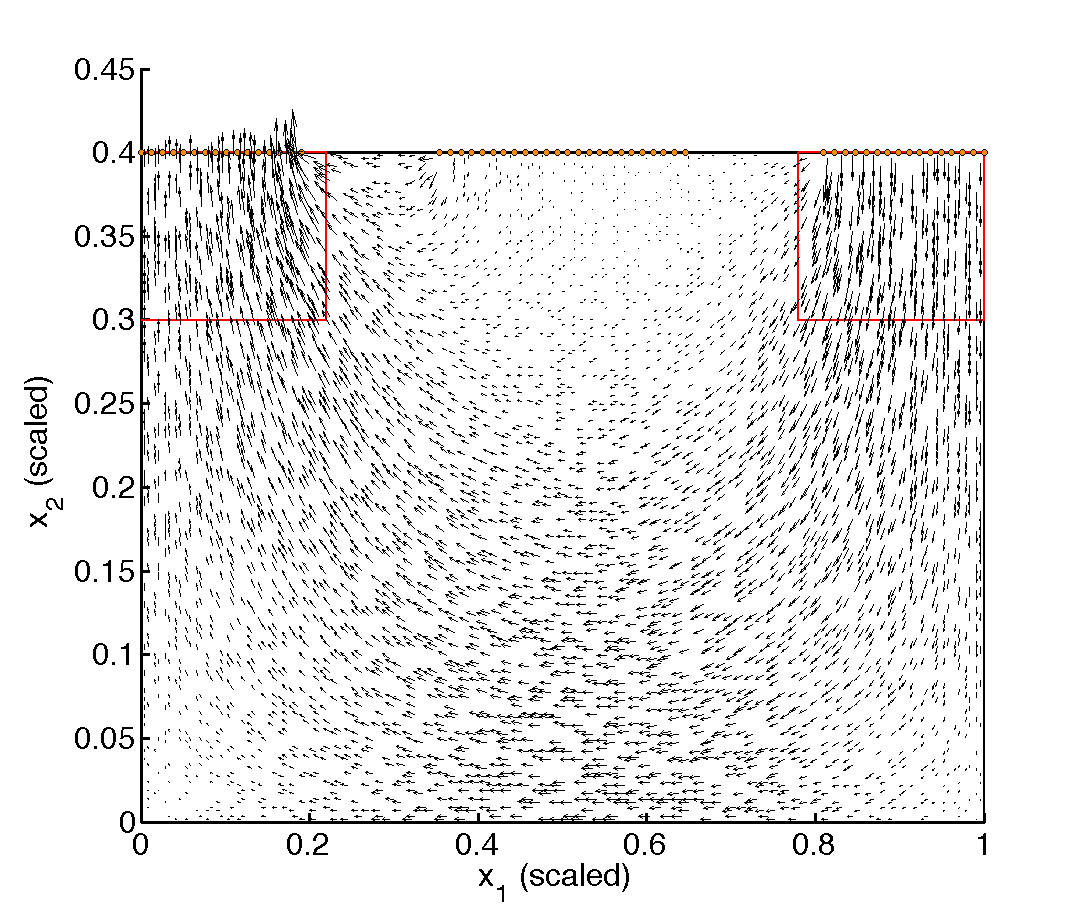}
    \label{fig:CurrentoptQDDgammaQDD}
    }

   \subfigure[Optimized doping profile of DD.]{
   \includegraphics[scale=0.15]{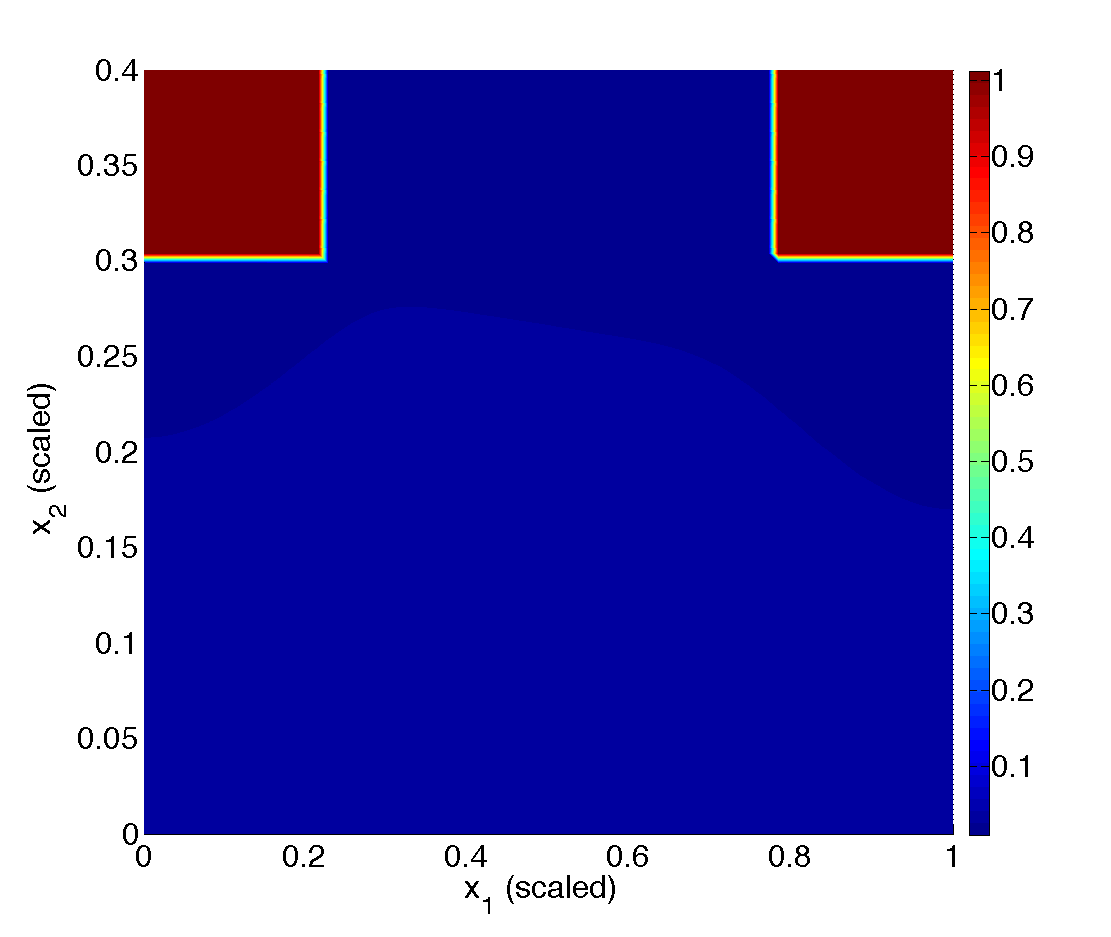}
   \label{fig:DopingoptDDgammaQDD}
   }
   \hspace*{0.3em}
   \subfigure[Total current density for optimized doping profile of DD ($\|n\n S\|_\infty = 0.702$).]{
   \includegraphics[scale=0.15]{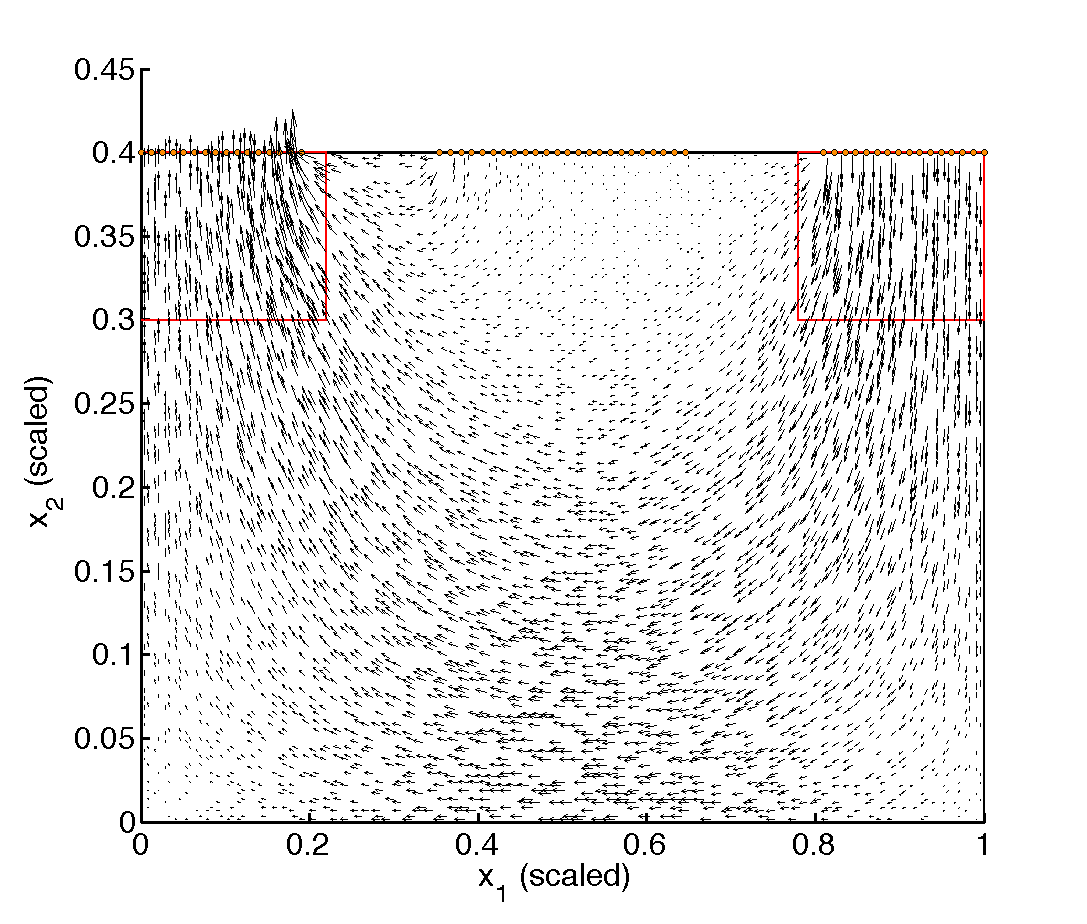}
   \label{fig:CurrentoptDDgammaQDD}
   }
          
  \caption{Total current densities (orange dots represent source and gate, the red rectangle the high doped $n_+$ region) and the reference and optimized doping profiles in the QDD and DD.  For the optimization we choose the regularization parameter $\gamma=1$.}
  \label{fig:gammaQDDopt}
\end{figure}

\begin{figure}[htbp]
\centering  
         \subfigure[Relative deviation of QDD optimized doping profile from reference doping profile.]{
         \includegraphics[scale=0.15]{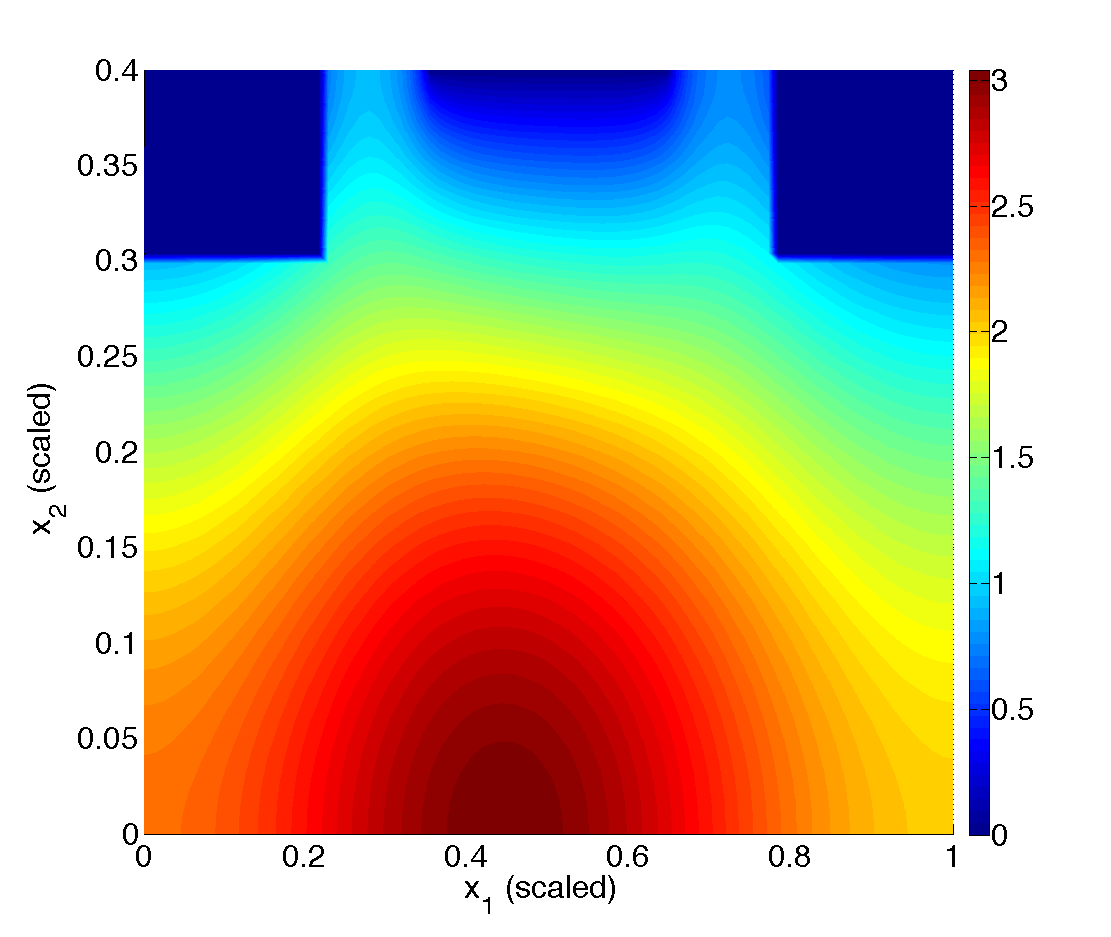}
         \label{fig:gammaQDDreldeviationDopingProfileQDDReference}
         }
         \hspace{0.3em}
   	 \subfigure[Relative deviation of QDD optimized doping profile from DD optimized doping profile.]{
   	 \includegraphics[scale=0.15]{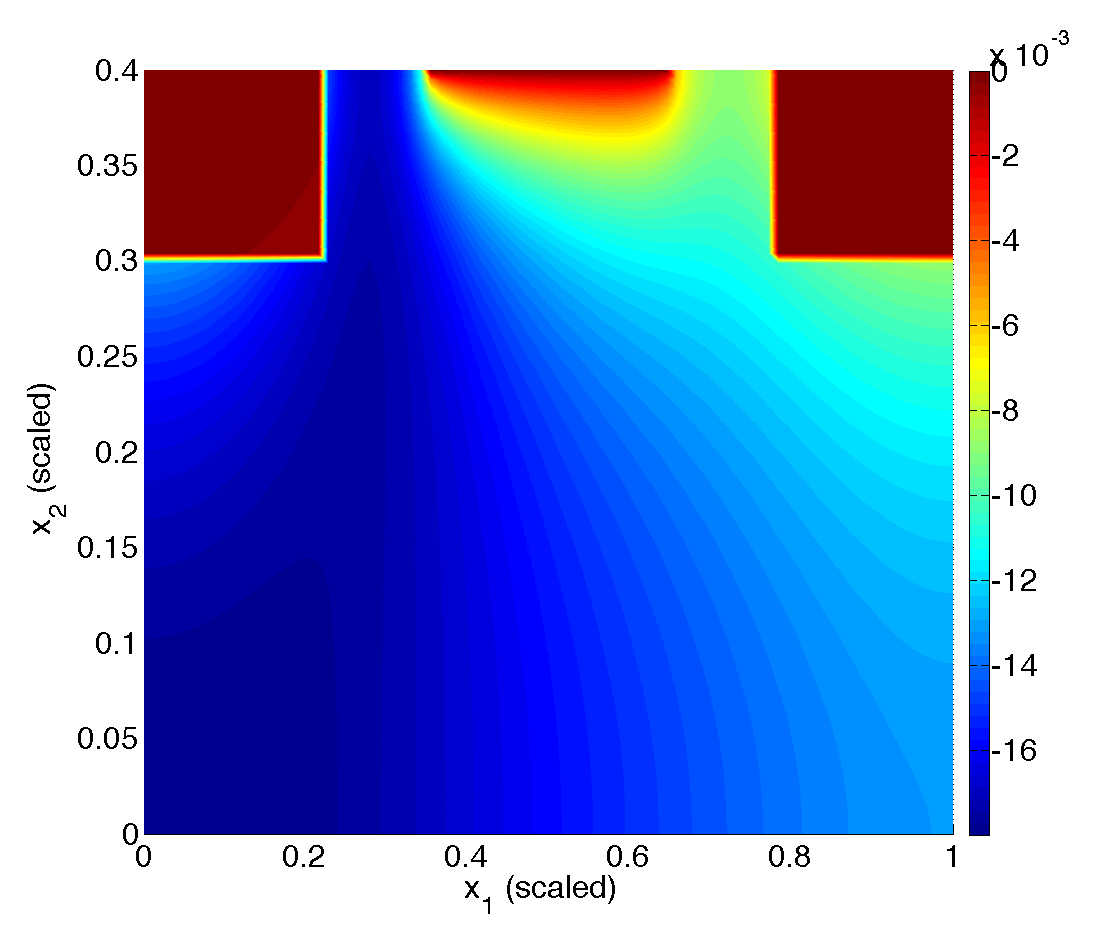}
   	 \label{fig:gammaQDDreldeviationDopingProfileQDDDD}
   	 }
         \subfigure[Relative deviation of the electron density for QDD optimized state from QDD reference state.]{
         \includegraphics[scale=0.15]{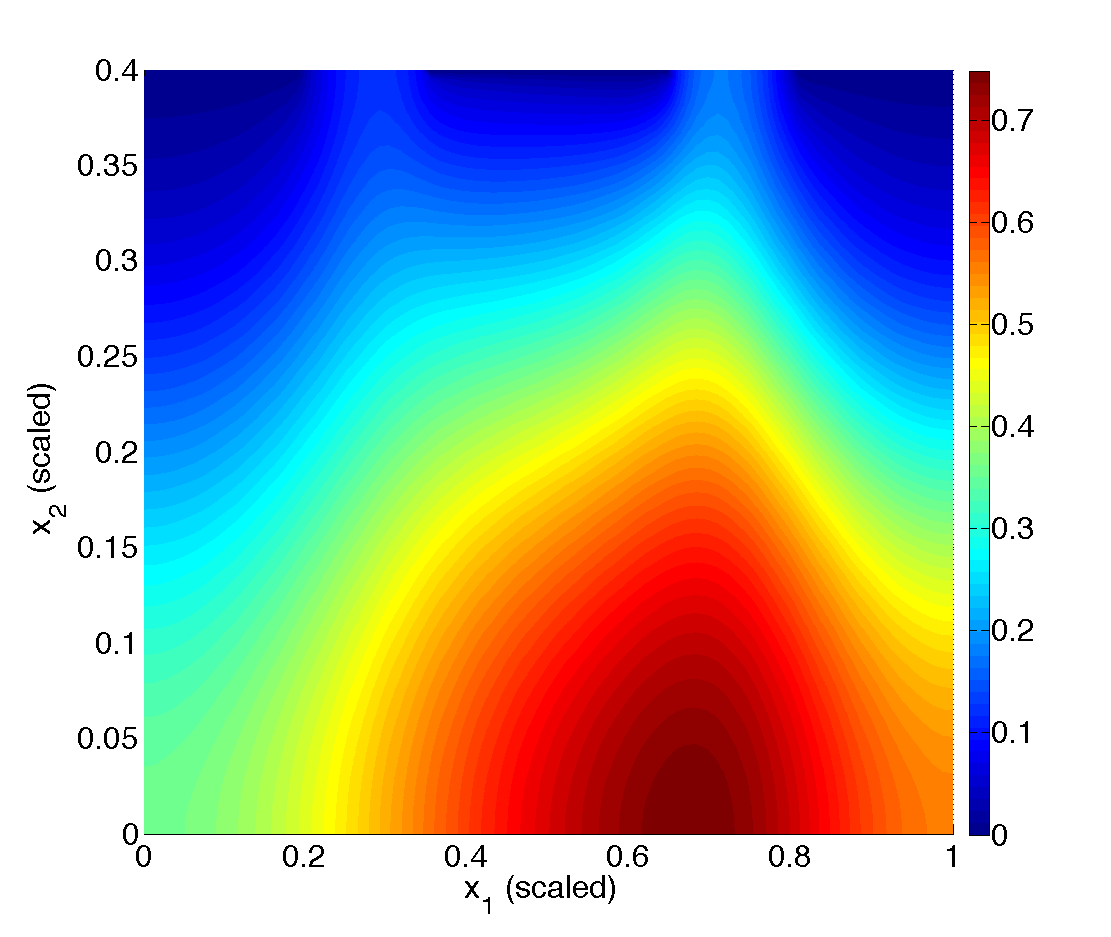}
         \label{fig:gammaQDDreldeviationDensityQDDReference}
         }
         \hspace{0.3em}
         \subfigure[Relative deviation of the electron density for QDD optimized state from DD optimized state.]{
         \includegraphics[scale=0.15]{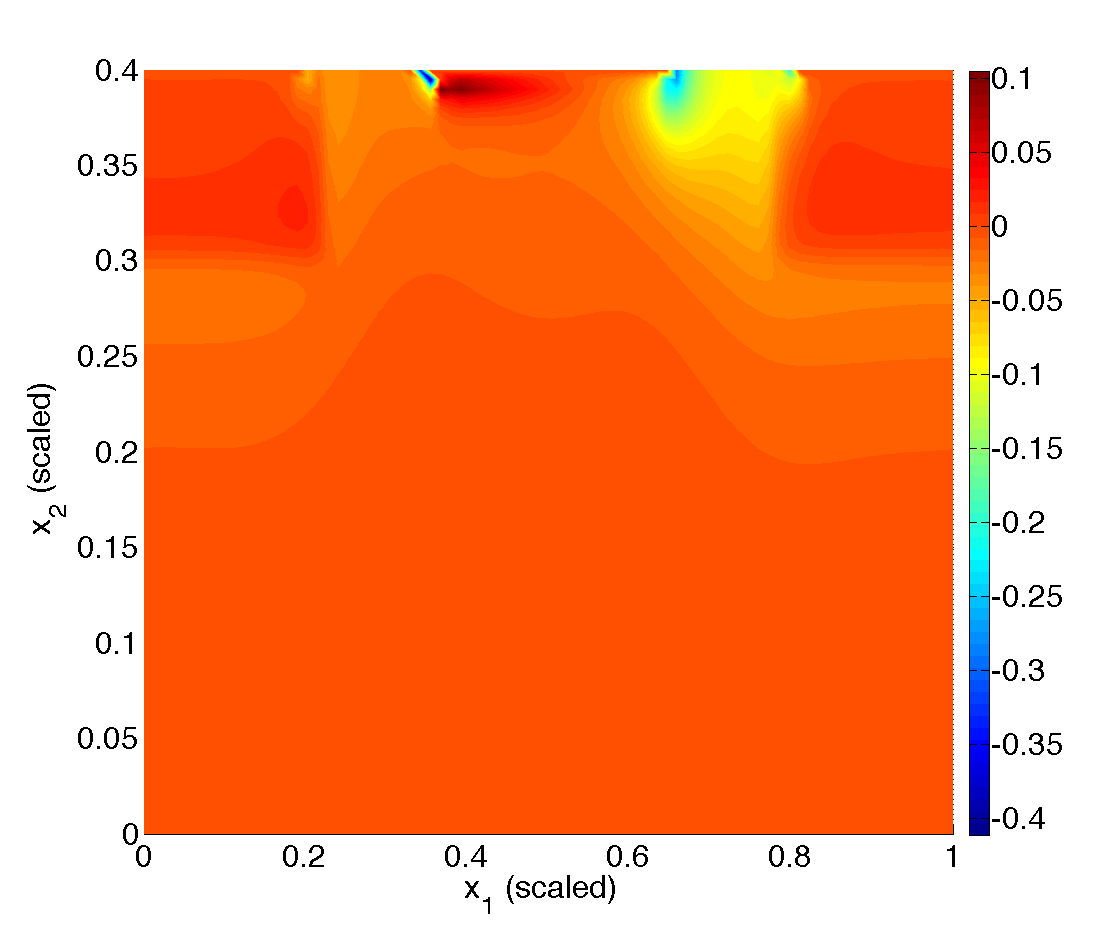}
         \label{fig:gammaQDDreldeviationDensityQDDDD}
         }
  \caption{Relative deviation of optimized doping profiles and electron densities to the reference solutions between QDD and DD.}
  \label{fig:gammaQDDreldeviation}
\end{figure}

Now we investigate the semi-classical limit $\e \rightarrow 0$ numerically in more detail. Let $\e_k = \e \. 10^{-n}, n=0,...,5$.
Since the solutions found by Algorithm \ref{alg:GradientOptimization} might only be local minimizers and minima instead of global ones, Proposition~\ref{convminima} and Corollary~\ref{convminimizers} do not necessarily require them to converge.
Nevertheless, from Fig.\,\ref{fig:gammaQDDConvergence} we see that this is the case and that they converge to the output of Algorithm~\ref{alg:GradientOptimization} for the classical model. This might be some indication that we have found global minimizers and minima.

\begin{figure}
\centering
      \subfigure{
      \includegraphics[scale=0.15]{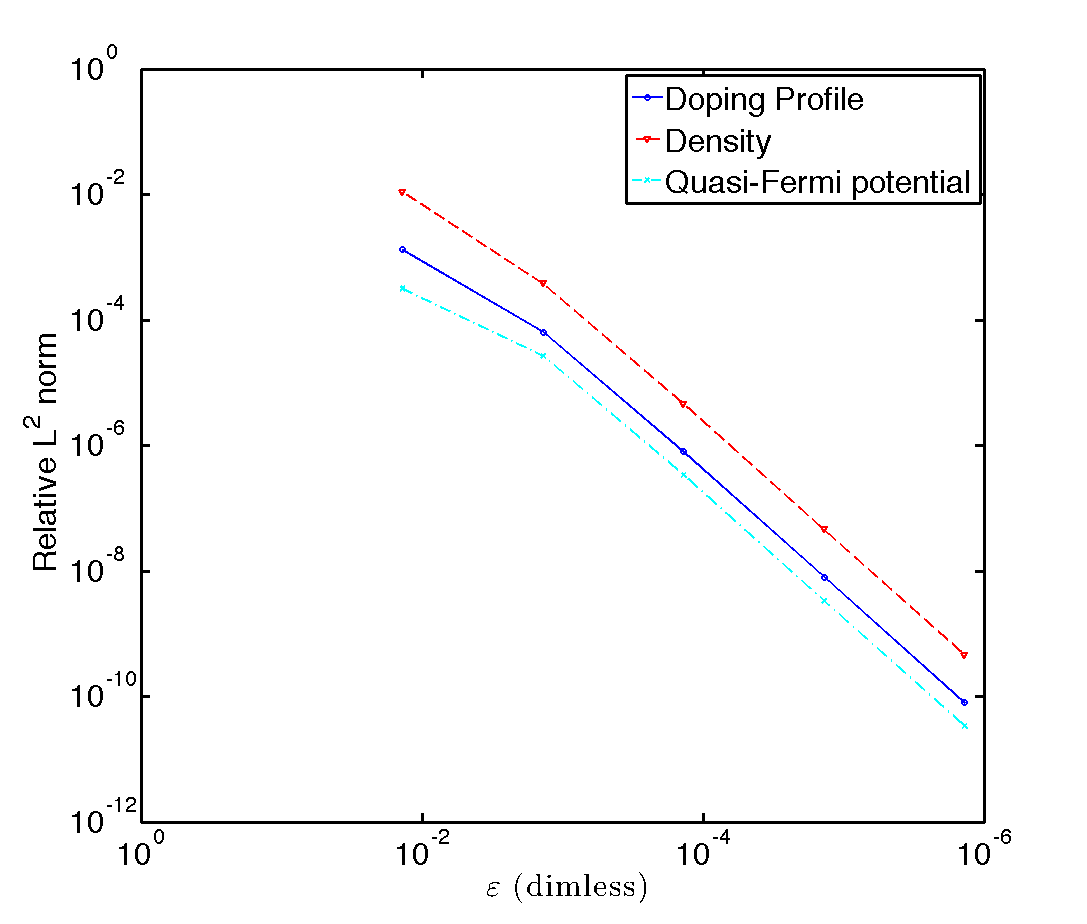}
      \label{fig:gammaQDDL2convergence}
      }
      \hspace{0.3em}
      \subfigure{
      \includegraphics[scale=0.15]{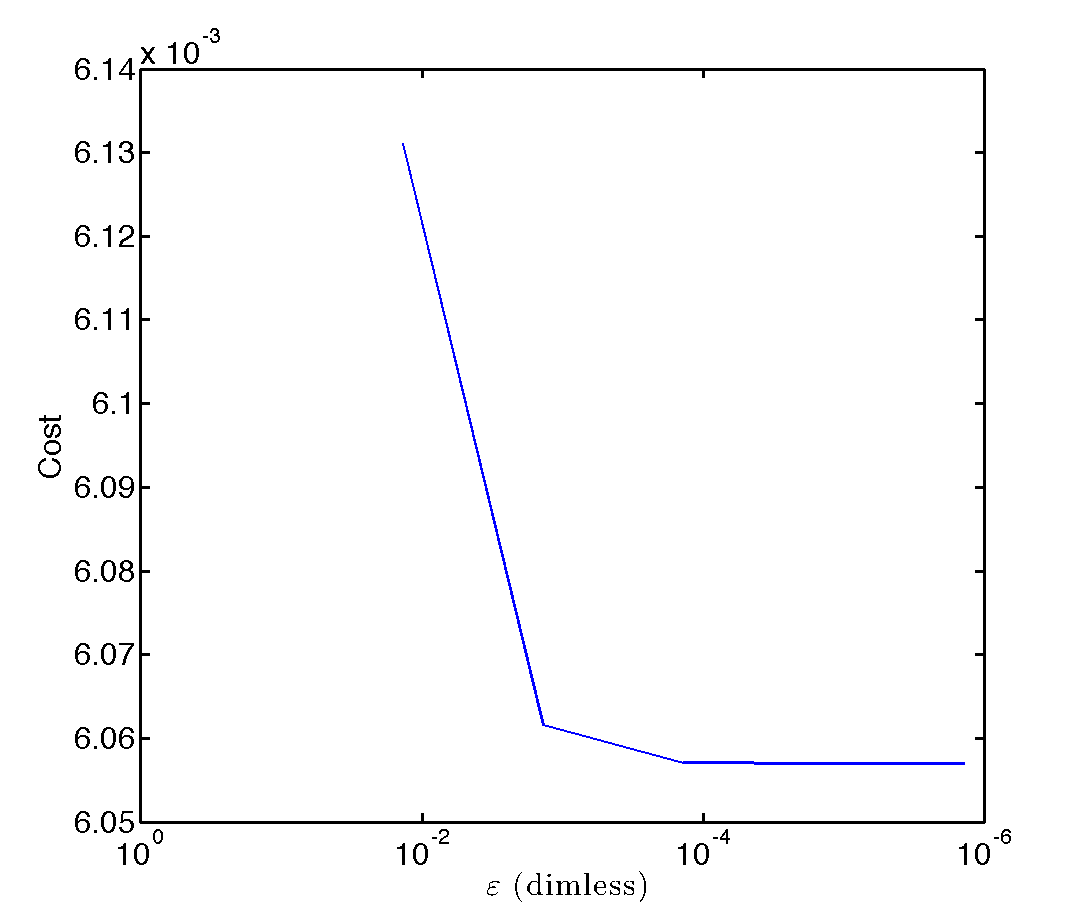}
      \label{fig:gammaQDDcostconvergence}
      }
  \caption{Relative $L^2$-norm of doping profile, electron density and quasi Fermi potential (left) and optimal cost (right) for $\e \rightarrow 0$. }
  \label{fig:gammaQDDConvergence}
\end{figure}

\section{Conclusions}
Using the concept of $\Gamma$-convergence we showed that the optimal semiconductor design problem for the quantum drift diffusion model is robust in the semiclassical limit. 
The numerical results even suggest that a stronger convergence might hold, at least in the case of unique solvability of the state equations. Hence, for small scaled Planck constants already the optimization of the more simple classical drift diffusion model gives adequate results also for the quantum model. This allows to  reduce significantly the numerical design costs.

\section*{Acknowledgements}
The authors acknowledge support from the DFG via SPP 1253/2.

\appendix

\section{}\label{proof:lemmawbound}

\begin{proof}[Proof of Lemma~\ref{lemmawbound}]
The idea of the proof is to derive uniform bounds for the sequences $(\rho_k)$, $(S_k)$ and $(V_k)$ in $L^{\infty}(\Omega)$. From this we derive a uniform bound for $\| S_k \|_{H^1(\Omega)}$, and by using the energies $E^{\e_k}_{S_k}(\rho)$ and $E^0_S(\rho)$ we derive the uniform boundedness of $\| \rho_k \|_{H^1(\Omega)}$.

So let $(\e_k)$ and $(\rho_k)$ be a sequences with the required properties. Using a cut-off operator, uniform bounds in $k$ for $\| S_k \|_{L^{\infty}(\Omega)}$ were shown in \cite{AbdUnt}. Consequently, 
\[
\| S_k \|_{L^{\infty}(\Omega)} \leq M_1 \quad \text{for all}\;\; k \in \nn,
\]
for some constant $M_1 > 0$.
This yields a uniform bound for $\inf_{\rho \in \Y_1}E^{\e_k}_{S_k}(\rho)$. As a consequence of the asymptotic growth $h(t) \in \mathcal{O}(t^{\frac{4}{3}})$ for $t \rightarrow \infty$, we further obtain the uniform boundedness of $\| \rho_k \|_{L^{10/3}(\Omega)}$, which, together with the boundedness of the sequence $(C_k)$ in $\mathcal{C}$, lead to a uniform bound for $(V_k)$ (cf.~\cite[Section~3.2]{MeyPhiTro}),
\[
 \| V_k \|_{H^1(\Omega)} + \| V_k \|_{L^{\infty}(\Omega)} \leq M_2\quad \text{for all}\;\; k \in \nn,
\]
for some constant $M_2 > 0$.

To derive a uniform lower bound of $\rho_k$, we multiply equation (\ref{QDD}a) with the test function $\test_k=\min\{0,\rho_k-\underline{\rho}\}$, where $\underline{\rho} > 0$ is a constant to be chosen appropriately. Integration by parts yields
\begin{align}
\label{QDD:uniformlowerboundw}
\e^2_k \int_{\Omega} | \n \test_k|^2 \ dx &= - \int_{ \{ \rho_k\, \leq \, \underline{\rho} \} }\rho_k\test_k \left( h(\rho_k^2)+ V_k - S_k \right) dx \\ 
& \leq - \int_{ \{ \rho_k\, \leq \, \underline{\rho} \} } \rho_k \test_k\left( h( \underline{\rho}^2) +  \overline{V} - \underline{S} \right)dx \nonumber 
\end{align}
with $ \overline{V} := M_2$ and $\underline S := \min\{ 0, \inf_{x \in \Omega, k \in \nn} S_k(x)\}>-\infty$, since $(S_k)$ is uniformly bounded in $L^{\infty}(\Omega)$. Now we may choose $\underline{\rho} > 0$ such that
\[
h( \underline{\rho}^2) +  \overline{V} - \underline{S} = 0
\]
holds. Therefore, the right-hand side of (\ref{QDD:uniformlowerboundw}) is equal to zero. Due to the boundary conditions for (\ref{QDD}a) we obtain
\[
\test_k \equiv 0 \quad \text{a.e.~in}\;\; \Omega,\; \forall\,k \in \nn\quad 
\Longrightarrow\quad  \rho_k \geq \underline{\rho}\quad \text{a.e.~in}\;\; \Omega,\;\forall\, k \in \nn. 
\]
Analogously, we prove the upper bound $\overline{\rho}$. Altogether we obtain
\[
\underline{\rho} \leq \rho_k \leq \overline{\rho} \quad\text{a.e.~in}\;\;\Omega,\;\;\text{for all}\;\; k \in \nn.
\]
Since $S_k = \Phi_S[\rho_k^2]$, we directly infer 
\[
 \| S_k \|_{H^1(\Omega)} \leq M_3 \quad\text{for all}\;\; k \in \nn,
\]
for some constant $M_3 > 0$.
 
 It remains to show the uniform $H^1(\Omega)$ bound for $(\rho_k)$. For some fixed $k \in \nn$ and $S_k \in \Y_3$ we define the auxiliary system
 \begin{subequations}
 \label{QDD:auxiliarysystem}
 \begin{align}
 \label{QDD:auxiliary1}
 0 & = h(\tilde{\rho}_k^2) + \tilde{V}_k - S_k, \\
 -\lambda^2 \Delta \tilde{V}_k & = \tilde{\rho}_k^2 - C
 \end{align}
 on $\Omega$ with boundary conditions
 \begin{align*}
 \tilde{\rho}_k = \rho_D, \;\; \tilde{V}_k = V_D \quad \text{ on } \Gamma_D, \quad \partial_\nu \tilde{\rho}_k = \partial_\nu \tilde{V}_k = 0 \quad \text{ on } \Gamma_N.
 \end{align*}
 \end{subequations}
This means that we solve the classical drift-diffusion model for $(\tilde{\rho}_k,\tilde{V}_k)$ with the quantum Fermi potential $S_k$.  Note that $(\tilde{\rho}_k, \tilde{V}_k, C)$ solves the system (\ref{QDD:auxiliarysystem}) weakly if and only if $\tilde{\rho}_k$ is the unique minimizer of $E^0_{S_k}$ in $\Y_1$.  From \cite{AbdUnt} we know that for each Fermi potential $S_k \in \Y_3$ there exists a minimizer $\tilde{\rho}_k$ in $\Y_1$. Furthermore, we can find some constant $K\ge 1$ depending on $\| S_k\|_{L^{\infty}(\Omega)}$ such that
\begin{eqnarray}\label{QDD:auxiliaryuniformbound}
 1/K \leq \tilde{\rho}_k \leq K\quad\text{for all}\;\; k \in \nn.
\end{eqnarray}
Recall that $\| S_k\|_{L^{\infty}(\Omega)}$ is uniformly bounded. Due to the fact that $\rho_k$ is a minimizer of $E^{\e_k}_{S_k}$, we may estimate
\begin{align*}
\e_k^2 \int_{\Omega} |\n \rho_k|^2\, dx + E^0_{S_k}(\rho_k) & = E^{\e_k}_{S_k}(\rho_k) \leq E^{\e_k}_{S_k}(\tilde{\rho}_k) \\
&\hspace*{-4em} = \e_k^2 \int_{\Omega} |\n \tilde{\rho}_k|^2\ dx + E^0_{S_k}(\tilde{\rho}_k) \leq \e_k^2 \int_{\Omega} |\n \tilde{\rho}_k|^2\ dx + E^0_{S_k}(\rho_k),
\end{align*}
thereby impliying that
\[
 \int_{\Omega} | \n \rho_k|^2 \ dx \leq \int_{\Omega} |\n \tilde{\rho}_k|^2\ dx.
\]
It remains to derive a uniform bound for $(\tilde{\rho}_n)$ in $H^1(\Omega)$. 

Due to (\ref{QDD:auxiliaryuniformbound}) and Assumption \ref{gammaQDD:assenthalpyasylimit}, we can define $h'_0 := \essinf_{x \in \Omega} h'(\tilde \rho_k^2(x))$ with $h'_0\in(0,\infty)$. We then differentiate (\ref{QDD:auxiliary1}) and multiply it with $\n \tilde \rho_k/2 \tilde \rho_k h'(\tilde \rho_k^2)$ (note that $\tilde \rho_k$ and $h'$ are uniformly bounded in $k$ away from zero). Integration yields
\begin{align*}
\int_{\Omega} | \n \tilde{\rho}_h|^2  dx & = \int_{\Omega}\frac{1}{2 \tilde \rho_k h'(\tilde \rho_k^2)} \n \tilde \rho_k\.\n(  S_k - \tilde V_k) \, dx \\
& \leq \mu \int_{\Omega}  \n\tilde \rho_k \. \n (S_k - \tilde{V}_k)\, dx \leq \mu\|\n \tilde \rho_k\|_{L^2(\Omega)} \|\n (S_k - \tilde{V}_k) \|_{L^2(\Omega)} \\
& \leq \mu \|\n \tilde{\rho}_k \|_{L^2(\Omega)} \left( \| \n S_k \|_{L^2(\Omega)} +  \|\n \tilde{V}_k \|_{L^2(\Omega)} \right),
\end{align*}
with $\mu=K/2 h'_0$ simply due to H\"older's inequality. With (\ref{QDD:auxiliaryuniformbound}) and standard elliptic estimates \cite{renardy}, we obtain the uniform boundedness of $\|\n \tilde{V}_k \|_{L^2(\Omega)}$. 

Along with the uniform bounds on $\| \n S_k \|_{L^2(\Omega)}$, this yields the existence of a constant $M_4 >0$ such that
\[
 \| \tilde{\rho}_k \|_{H^1(\Omega)} \leq M_4 \quad \text{for all}\;\; k \in \nn,
\]
thereby infering the existence of yet another constant $M_5 >0$ with
\[
\| \rho_k \|_{H^1(\Omega)} \leq M_5 \quad \text{for all}\;\; k \in \nn,
\]
which finally concludes to proof.
\end{proof}

\section{}\label{proof:QDD:lemmaexistenceconvergentsolutions}

We will use a variant of the implicit function theorem to facilitate the proof.

\begin{proposition}\label{appendix:prop:regular}
 Let $\e_0>0$ and $F\colon Y\times[0,\e_0)\to Z$ be a mapping, where $Y$ and $Z$ are Banach spaces. Suppose 
 \begin{enumerate}
  \item[\em (i)] there exists $y_0\in Y$ satisfying $F(y_0,0)=0$,
  \item[\em (ii)] $F$ is Fr\'echet differentiable in a neighborhood of $(y_0,0)$ such that the remainder term $R(\eta,\e) = F(y+\eta,\e)-F(y,\e)-D_yF(y,\e)[\eta]$ satisfies
  \[
   \|R(\eta,\e)-R(\xi,\e)\|_{Z} \le L\delta \|\eta-\xi\|_{Y}
  \]
  in a neighborhood of $(y_0,0)$ for any $\eta,\xi\in Y$ with $\|\eta\|_Y,\|\xi\|_Y\le \delta$ for some constants $L>0$ and $\delta>0$ independent of $\e$, and
  \item[\em (iii)] the derivative $D_y F(y_0,0)$ has a bounded inverse, i.e.,
  \[
   \|D_y F(y_0,0)^{-1}\eta\|_Y \le K\|\eta\|_{Z}\quad\text{for all}\;\;\eta\in Z.
  \]
  Then, for sufficiently small $\e>0$, the problem $F(y_\e,\e)=0$ has a unique solution $y_\e\in Y$ in an $\e$-independent neighborhood of $y_0$, and
  \[
   \|y_\e - y_0\| = \mathcal{O}(\e)\quad\text{for}\;\;\e\to 0.
  \]
 \end{enumerate}
\end{proposition}

\begin{proof}[Proof of Lemma~\ref{QDD:lemmaexistenceconvergentsolutions}]  Let $\rho_0\in \Y_1$ be an isolated solution of the classical drift-diffusion equations with a fixed doping profile $C\in \mathcal{C}$. Further, let $\e_0>0$, and
\[
 Y := [\Y_0]^3,\qquad Z:= [H_0^1(\Omega\cup\Gamma_N)^*]^3\times [H^{\frac{1}{2}}(\Gamma_D)]^3.
\]
%
Consider the operator equation
\begin{equation}\label{appendix1:operator}
 F(y_{\e},\e)=0\quad\text{in}\;\; Z,
\end{equation}
where $y_\e=(\rho_\e,V_\e,S_\e)$ and $F\colon Y\times[0,\e_0) \to Z$ is defined by
\begin{align*}
 \langle F_1(y,\e),\test_1\rangle &= \e^2\int_\Omega \nabla\rho\.\nabla\test_1\,dx + \int_\Omega \rho\left(h(\rho^2)+V-S\right)\test_1\,dx \\
 \langle F_2(y,\e),\test_2\rangle &= \lambda^2\int_\Omega \nabla V\.\nabla\test_2\,dx - \int_\Omega (\rho^2-C)\test_2\,dx \\
 \langle F_3(y,\e),\test_3\rangle &= \int_\Omega \rho^2\nabla S\.\nabla\test_3\,dx
\end{align*}
for all $\test_1,\test_2,\test_3\in H_0^1(\Omega\cup\Gamma_N)$, and
\[
 F_4(y,\e) = \tau_D(\rho-\rho_D),\quad F_5(y,\e) = \tau_D(V-V_D),\quad F_6(y,\e) = \tau_D(S-S_D)
\]
where $\tau_D\colon H_0^1(\Omega\cup\Gamma_N)\to H^{\frac{1}{2}}(\Gamma_D)$ denotes the trace operator onto $\Gamma_D\subset \Gamma$. 

Notice that the operator equation above is equivalent to the one given in (\ref{weakformulation1}), and the equation with $\e=0$  corresponds to the classical drift-diffusion equations. Hence, $y_0=(\rho_0, V_0=\Phi_V[\rho_0^2-C],S_0=\Phi[\rho_0^2])\in Y$ satisfies $F(y_0,0)=0$ in $Z$. We also recall from \cite{UntVol} the Fr\'echet differentiability of the operator $F$ in a neighborhood of $(y_0,0)$. Furthermore, due to Assumption~\ref{gammaQDD:assclassicalsolutions} we deduce that the derivative $D_y F(y_0,0)$ has a bounded inverse, i.e.,
\[
 \|D_y F(y_0,0)^{-1}\eta\|_\Y \le K\|\eta\|_{Z}\quad\text{for all}\;\;\eta\in Z.
\]
The remainder term $R$ has the form
\begin{align*}
 \langle R(\eta,\e),\test\rangle &= \int_\Omega \left[2\rho_0 \eta_\rho\nabla \eta_S + \eta_\rho^2\,\nabla(S+\eta_\rho)\right]\.\nabla\test_1\,dx \\
 &\hspace*{-3em}+ \int_\Omega \left[\rho_0 r_\rho(\eta_\rho) + \eta_\rho\big(h(\rho_0 + \eta_\rho) - h(\rho_0) + \eta_V - \eta_S\big)\right]\test_2\,dx - \int_\Omega \eta_\rho^2\test_3\,dx,
\end{align*}
where $r_\rho$ is as defined in Assumption~\ref{gammaQDD:assenthalpyasylimit}. It remains an easy exercise to check that
\[
 \|R(\eta,\e)-R(\xi,\e)\|_{Z} \le L\delta \|\eta-\xi\|_{\Y}
\]
for any $\eta,\xi\in \Y$ with $\|\eta\|_\Y,\|\xi\|_\Y\le \delta$ for some constants $L>0$ and $\delta>0$. Observe that these constants are independent of $\e$, since $R$ is independent of $\e$.  We conclude the proof by applying Proposition~\ref{appendix:prop:regular}.
\end{proof}

\bibliography{biblio}
\bibliographystyle{abbrv}

\end{document}